\documentclass[12pt]{amsart}

\newif\ifps
\pstrue

\usepackage{amsmath, amsfonts, latexsym, array, amssymb}
\usepackage{enumitem}
\usepackage{psfrag}
\usepackage[all]{xypic}
\usepackage[all]{xy}
\usepackage{graphicx}
\usepackage[svgnames]{xcolor}
\usepackage{epstopdf}
\usepackage{hyperref}

\newtheorem{thm}{Theorem}[section]
\newtheorem{cor}[thm]{Corollary}
\newtheorem{lem}[thm]{Lemma}
\newtheorem{prop}[thm]{Proposition}

\newtheorem{defn}[thm]{Definition}

\newtheorem{thma}{Theorem}

\def\epsilon{\varepsilon}

\newcommand{\occult}[1]{}

\newcommand{\foot}[1]{\mbox{}\marginpar{\raggedleft\hspace{0pt}\tiny #1}}

\newcommand{\old}[1]{}

\usepackage{ulem}

\newcommand{\htop}{h_\mathrm{top}}
\newcommand{\eps}{\varepsilon}
\newcommand{\ph}{\varphi}

\newcommand\RR{{\mathbb R}}

\newcommand\TT{{\mathbb T}}
\newcommand\ZZ{{\mathbb Z}}
\newcommand\NN{{\mathbb N}}
\newcommand{\mm}{\mathbf{m}}
\newcommand{\nn}{\mathbf{n}}

\newcommand{\Lspan}{\Lambda^\mathrm{span}}
\newcommand{\Lsep}{\Lambda^\mathrm{sep}}

\newcommand{\CCC}{\mathcal{C}}
\newcommand{\GGG}{\mathcal{G}}

\newcommand{\JJJ}{\mathcal{J}}
\newcommand{\DDD}{\mathcal{D}}

\newcommand{\SSS}{\mathcal{S}}
\newcommand{\PPP}{\mathcal{P}}

\newcommand{\UUU}{\mathcal{F}}
\newcommand{\VVV}{\mathcal{V}}

\newcommand{\Pexp}{P_\mathrm{exp}^\perp}

\newcommand{\bk}{\bar\kappa}
\newcommand{\phigeo}{\varphi^{\mathrm{geo}}}
\newcommand{\twentyone}{{100}} 
\newcommand{\sixtythree}{{300}} 
\newcommand{\onetwosix}{{600}} 
\newcommand{\fbv}{f_{BV}}
\newcommand{\fB}{f_A}
\DeclareMathOperator{\Diff}{Diff}
\DeclareMathOperator{\Var}{Var}
\DeclareMathOperator{\Id}{Id}

\numberwithin{equation}{section}

\begin{document}

\title[Unique equilibrium states]{Unique equilibrium states for Bonatti--Viana diffeomorphisms}
\author[V.~Climenhaga]{Vaughn Climenhaga}
\author[T.~Fisher]{Todd Fisher}
\author[D.~J.~Thompson]{Daniel J. Thompson}
\address{Department of Mathematics, University of Houston, Houston, TX 77204}
\email{climenha@math.uh.edu}
\address{Department of Mathematics, Brigham Young University, Provo, UT 84602}
\email{tfisher@mathematics.byu.edu}
\address{Department of Mathematics, The Ohio State University, 100 Math Tower, 231 West 18th Avenue, Columbus, Ohio 43210}
\email{thompson@math.osu.edu}

\thanks{V.C.\ is supported by NSF grant DMS-1554794.  T.F.\ is supported by Simons Foundation grant \# 239708. D.T.\ is supported by NSF grant  DMS-$1461163$}

\date{\today}
\commby{}

\begin{abstract}
We show that  the robustly transitive diffeomorphisms constructed by Bonatti and Viana have unique equilibrium states for natural classes of potentials. 
In particular, we characterize the SRB measure as the unique equilibrium state for a suitable geometric potential. The techniques developed are applicable to a wide class of DA diffeomorphisms, and persist under $C^1$ perturbations of the map. These results are an application of general machinery developed by the first and last named authors.
 \end{abstract}
\maketitle

\section{Introduction and statement of results}
\normalem

An \emph{equilibrium state} for a diffeomorphism $f\colon M\to M$ and a potential $\ph \colon M \to \RR$ is an invariant Borel probability measure that maximizes the quantity $h_\mu(f) + \int \ph\,d\mu$.  Results on existence and uniqueness of equilibrium states have a long history \cite{rB75, DKU, Ho2, BS03, BT09, CT1,  PPS, PSZ}, and are one of the main goals in thermodynamic formalism. 
Such results are a powerful tool to understand the orbit structure and global statistical properties of dynamical systems, and often lead to further applications, including large deviations principles, central limit theorems, and knowledge of dynamical zeta functions \cite{PP}.

The benchmark result of this type is that there is a unique equilibrium state $\mu$ when $(M,f)$ is uniformly hyperbolic, mixing, and $\ph$ is H\"older continuous.  
When $\ph$ is the \emph{geometric potential} $\varphi(x) = -\log \mathrm{det}(Df|_{E^u}(x))$, 
this unique equilibrium state is the SRB measure \cite{jS72,rB75,dR76}. Extending this type of result beyond uniform hyperbolicity is a major challenge in the field.  The first and third authors have developed techniques to establish existence and uniqueness of equilibrium states in the presence of non-uniform versions of specification and expansivity \cite{CT4}. 
The purpose of this paper is to show how these results can be applied to higher dimensional smooth systems with weak forms of hyperbolicity, where alternative approaches based on symbolic dynamics or transfer operators appear to meet with fundamental difficulties.  While thermodynamic formalism for one-dimensional systems is well developed, and there have been major recent breakthroughs in dimension two by Buzzi, Crovisier, and Sarig \cite{oS13, BCS}, 
 the higher-dimensional case remains poorly understood.

{

We focus on the class of Bonatti--Viana diffeomorphisms \cite{BV}; these are robustly transitive, derived from Anosov (DA), diffeomorphisms of $\TT^4$.  These are the model examples of robustly transitive DA systems with a dominated splitting that are not partially hyperbolic.  This setting demonstrates the flexibility of our methods while having the advantage of being concrete. To the best of our knowledge, no other techniques for uniqueness are available for the class of Bonatti--Viana diffeomorphisms, or indeed any natural class of smooth systems on $\TT^4$ beyond uniform hyperbolicity.
 
}

 The Bonatti--Viana construction \cite{BV, BF} is a $C^0$ perturbation of a $4$-dimensional toral automorphism $\fB$ with
a hyperbolic splitting $E^s \oplus E^u$, where $\dim E^s = \dim E^u=2$.
 The perturbation 
can be characterized by two parameters $\rho$ and $\lambda$: the parameter  $\rho>0$ is the size of the balls $B(q, \rho) \cup B(q', \rho)$ inside which the perturbation takes place, where $q, q'$ are fixed points; the parameter $\lambda>1$ is the maximum of expansion in the centre-stable and expansion in backwards time in the centre-unstable. 
The construction can be carried out with both $\rho$ and $\log \lambda$ arbitrarily small if required. 

For fixed $\lambda>1$ and $\rho >0$, we write $\fbv \in \UUU_{\lambda, \rho}$ for a diffeomorphism provided by the Bonatti--Viana construction for which these parameters are bounded above by these values of $\lambda$ and $\rho$. It is crucial for our analysis that as a DA system, $f_{BV}$ has a dominated splitting $E^{cs} \oplus E^{cu}$ that is close to the Anosov splitting for $\fB$, and that $f_{BV}$ displays uniformly hyperbolic behaviour outside the neighbourhood on which the perturbation takes place.

Our results give quantitative criteria for existence and uniqueness of the equilibrium state in terms of a function $\Phi$ which depends on the size of the perturbation from the original Anosov map (via the parameters $\rho$ and $\lambda$), the norm and variation of the potential, and the tail entropy of the system. 
The idea is that $\Phi$ gives an upper bound for the pressure of the ``non-hyperbolic'' part of the system, so that for a potential $\ph$ and a diffeomorphism $g\approx f_{BV}$, if $\Phi(\ph; g)$ is smaller than the \emph{topological pressure} $P(\ph; g):= \sup \{h_\mu(g) + \int\ph\,d\mu : \mu$ is $g$-invariant$\}$, then the system ``sees enough uniformly hyperbolic behavior'' to establish uniqueness of the equilibrium state. The uniqueness comes from an application of the general machinery of \cite{CT4}, which is a non-uniform version of a classic proof of Bowen \cite{Bow75}.

\begin{thma}\label{main2}
Let $\fbv \in \UUU_{\lambda, \rho}$ be a diffeomorphism obtained by the Bonatti--Viana construction 
and $g$ be a $C^1$ perturbation of $\fbv$. Let $\varphi\colon \mathbb{T}^4\to \mathbb{R}$ be a H\"older continuous potential function. 
There is a function $\Phi = \Phi (\varphi; g)$, given explicitly at \eqref{def:phi}, such that
\begin{enumerate}[leftmargin=0.7cm, label=\upshape{(\arabic{*})}]
\item
$\lim_{\lambda \to 1, \rho \to 0} \sup \{ \Phi (\varphi; f) : f \in \UUU_{\lambda, \rho} \}= \max\{\varphi(q), \varphi(q')\}< P(\varphi; \fB)$;
\item $\Phi (\varphi; g)$ varies continuously under $C^0$ perturbation of the potential function and $C^1$ perturbation of the map;
\end{enumerate}
and if $\Phi (\varphi; g)  < P(\varphi; g)$, then $(\TT^4,g,\ph)$ has a unique equilibrium state.
\end{thma}

A more precise statement of this result, including the definition of $\Phi$, is given as Theorem \ref{t.BV}. In \S \ref{s:corproof}, by analyzing the function $\Phi$ and the topological pressure $P(\varphi; g)$, we give two corollaries of this result: in Corollary \ref{t.BV-range}, we show that for a fixed diffeomorphism $g$ which is a $C^1$ perturbation of $\fbv$, every H\"older continuous $\varphi$ satisfying a bounded range condition has a unique equilibrium state; in Corollary \ref{cor1.2}, we show that for a fixed H\"older continuous potential then there is a unique equilibrium state with respect to $\fbv$ as long as the parameters $\rho, \log \lambda$ in the Bonatti-Viana construction are sufficiently small.

Consider the potential function $\phigeo(x) =-\log \det(Df|_{E^{cu}(x)})$, where $E^{cu}(x) \subset T_x \TT^4$ is the two-dimensional centre-unstable subspace at $x$.  We refer to this as the \emph{geometric potential}; see for example \cite{rB75, IT10,GS14,C14} for the terminology and for applications of the family $\{t\phigeo : t\in \RR\}$ including multifractal analysis of Lyapunov exponents.

\begin{thma}\label{main3}
Let $\fbv \in \UUU_{\lambda, \rho}$ with $\log \lambda, \rho$ sufficiently small. Then for every $C^2$ diffeomorphism $g$ which is a sufficiently small $C^1$ perturbation of $\fbv$, the following are true.
\begin{itemize}
\item $t=1$ is the unique root of the function $t\mapsto P(t\phigeo_g; g)$.
\item There is an $\eps>0$ such that $t\phigeo_g$ has a unique equilibrium state $\mu_t$ for each $t\in (-\eps,1+\eps)$.
\item $\mu_1$ is the unique SRB measure for $g$.
\end{itemize}
\end{thma}
Our results are proved using general machinery developed by the first and last named authors~\cite{CT4}. The idea is to find a `good' collection of orbit segments on which the map has uniform expansion, contraction, and mixing properties, and demonstrating that this collection is large in the sense that any orbit segment can be \emph{decomposed} into `good' and `bad' parts in such a way that the collection of `bad' orbit segments has smaller topological pressure than the entire system.  

The diffeomorphisms we consider are not expansive (see \S\ref{sec:tail-entropy} for definitions of expansivity and related concepts).
In fact, a $C^1$ perturbation of a diffeomorphism $\fbv$ may not even be asymptotically h-expansive, and thus may have positive tail entropy \cite{BF}. We handle this by showing that any measure with large enough free energy is \emph{almost expansive} (Definition \ref{almostexpansive}), so the failure of expansivity does not affect equilibrium states. A significant technical point in our approach is that we carry out all our estimates coarsely at a definite fixed scale which is too large to `see' the bad dynamics that may occur at small scales or infinitesimally. We control the smaller scales indirectly using bounds on the tail entropy. The need to be precise and careful about scales leads to substantial technicalities in our arguments.

We now review the relevant results in the literature. For systems with a dominated splitting, there are some known results on uniqueness of the measure of maximal entropy
although these mostly require partial hyperbolicity \cite{Berg, UR, RHRHTU}, and the case of equilibrium states for $\ph\neq 0$ have been largely unexplored. 
For the Bonatti--Viana examples,  the existence of a unique MME was obtained in~\cite{BF}, {using a technique that is not suited to generalization to equilibrium states.} 
Existence of equilibrium states for partially hyperbolic horseshoes was studied by Leplaideur, Oliveira, and Rios \cite{LOR}, but they do not deal with uniqueness. {Results for uniqueness of equilibrium states for frame flows have been obtained recently by Spatzier and Visscher \cite{SV}.} Other references which apply in higher dimensional settings include \cite{BS03, PSZ}. In particular, Pesin, Senti and Zhang  \cite{PSZ} have used tower techniques to develop thermodynamic formalism for the Katok map, which is a non-uniformly hyperbolic DA map of the $2$-torus.

The theory of SRB measures has received much more attention. The fact that there is a unique SRB measure for the examples we study follows from~\cite{BV}. 
The connection between SRB measures and equilibrium states is given by the Ledrappier--Young formula and the Margulis--Ruelle inequality. 
However, even when there is known to be a unique SRB measure, the characterization of the SRB as an equilibrium state of a continuous potential function requires a non-trivial proof because the number of positive Lyapunov exponents can be different for different measures.  For diffeomorphisms with a dominated splitting, Carvalho \cite{mC93} has showed that the SRB measure for a DA system obtained along an arc of $C^\infty$ diffeomorphisms is an equilibrium state. Along a $C^r$ arc, her result only applied at first bifurcation. To the best of our knowledge, our results in \S\ref{s.srb} are the first that characterize the SRB measure as a \emph{unique} equilibrium state for a class of diffeomorphisms with a dominated splitting beyond uniform hyperbolicity.

The techniques introduced in this paper are robust and apply for other DA systems. In \cite{CFT_Mane}, we use the tools and ideas introduced in this paper to study  the partially hyperbolic Ma\~n\'e family of diffeomorphisms \cite{man78}. This family is significantly easier to study than the Bonatti--Viana family, and we are able to derive stronger results in that setting. Almost Anosov diffeomorphisms, Katok maps, and the Shub class of robustly transitive diffeomorphisms~\cite{HPS} are other classes of DA systems where these techniques can be explored.

For uniformly hyperbolic systems,
unique equilibrium states associated to H\"older potentials are known to have strong statistical properties, such as Bernoullicity, central limit theorem, and exponential decay of correlations.  The proofs of these properties use Markov partitions and quasi-compactness of the transfer operator, while our proofs are based on extending Bowen's approach via specification.  Ledrappier has established the K property for systems with uniform specification \cite{fL77} and statistical poperties for symbolic systems have been explored in \cite{C15}. It is an open problem whether specification-based approaches can be used to establish these stronger statistical properties in general. We expect that the unique equilibrium states produced here will have these stronger statistical properties, although this remains to be proved.

In \S\ref{s.back}, we give the necessary background material  from \cite{CT4} on thermodynamic formalism.  In \S\ref{s.perturb}, we prove general pressure estimates for $C^0$-perturbations of Anosov systems.
In \S\ref{sec:BV}, we provide details of the Bonatti--Viana construction, and state a more precise version of Theorem \ref{main2}. In \S\ref{s:corproof}, we prove corollaries of Theorem \ref{main2}.  In \S\ref{s:mainproof}, we prove our main theorem. 
In \S\ref{s.srb}, we prove Theorem \ref{main3} on SRB measures. In \S \ref{s.lemmas}, we provide proofs for a few technical lemmas. In an Appendix, we give a direct proof of the required regularity for the geometric potential.

\section{Background}\label{s.back}

In this section, we state definitions and results that we will need throughout the paper.  


\subsection{Pressure} 
Let $X$ be a compact metric space and $f\colon X\to X$ be a continuous map.  Henceforth, we will identify $X\times \mathbb{N}$ with the space of finite orbit segments for a map $f$ via the correspondence
\begin{equation}\label{eqn:segments}
(x,n) \quad\longleftrightarrow\quad (x,f(x),\dots,f^{n-1}(x)).
\end{equation}
For a continuous potential function $\varphi\colon X\to \mathbb{R}$ we write 
\[
S_n\varphi(x) = S_n^f \varphi(x) = \sum_{k=0}^{n-1} \varphi(f^kx)
\]
for the ergodic sum along an orbit segment, and given $\eta>0$, we write 
\[
\Var(\ph,\eta) = \sup\{|\ph(x)-\ph(y)| \,:\, x,y\in X, d(x,y) < \eta\}.
\]
Given $n\in \NN$ and $x,y\in X$, we write
\[
d_n(x,y) = \max \{ d(f^kx,f^ky) \,:\, 0\leq k < n\}.
\]
Given $x\in X$, $\eps>0$, and $n\in \NN$, the \emph{Bowen ball of order $n$ with center $x$ and radius $\eps$} is
\[
B_n(x,\eps) = \{y\in X \, :\,  d_n(x,y) < \eps\}.
\]
We say that $E\subset X$ is $(n,\eps)$-separated if $d_n(x,y) \geq \eps$ for all $x,y\in E$.

We will need to consider the \emph{pressure of a collection of orbit segments}.  More precisely, we interpret $\mathcal{D}\subset X\times \mathbb{N}$ as a collection of finite orbit segments, and write $\mathcal{D}_n = \{x\in X \, :\,  (x,n)\in \mathcal{D}\}$ for the set of initial points of orbits of length $n$ in $\mathcal{D}$.  Then we consider the partition sum
$$
\Lambda^{\mathrm{sep}}_n(\DDD,\ph,\epsilon; f) =\sup
\left\{ \sum_{x\in E} e^{S_n\ph(x)} \, :\,  E\subset \mathcal{D}_n \text{ is $(n,\epsilon)$-separated} \right\}.
$$

When there is no confusion in the map we will sometimes omit the dependence on $f$ from our notation.  
We will also sometimes require a partition sum $\Lspan_n$ defined with $(n,\eps)$-spanning sets. Given $Y\subset X$, $n\in \NN$, and $\delta>0$, we say that $E\subset Y$ is an \emph{$(n,\delta)$-spanning set for $Y$} if $\bigcup_{x\in E} \overline{ B_n(x,\delta)} \supset Y$.  Write
\[
\Lspan_n(\DDD,\ph,\delta;f) =\inf \left\{ \sum_{x\in E} e^{S_n \ph(x)} \, :\,  E\subset \DDD_n \text{ is $(n,\delta)$-spanning}\right\}.
\]

We will use the following basic result relating $\Lambda^{\mathrm{sep}}_n$ and $\Lspan_n$, which is proved in \S\ref{s.lemmas}.
\begin{lem}\label{lem:span-sep}
For any $\DDD \subset X\times \NN$, $\ph\colon X\to \RR$, and $\delta>0$, we have
\begin{align*}
\Lspan_n(\DDD,\ph,\delta) &\leq \Lambda^{\mathrm{sep}}_n(\DDD,\ph,\delta), \\
\Lambda^{\mathrm{sep}}_n(\DDD,\ph,2\delta) &\leq e^{n \Var(\ph,\delta)} \Lspan_n(\DDD,\ph,\delta).
\end{align*}
\end{lem}


The \emph{pressure of $\ph$ on $\DDD$ at scale $\eps$} is 
$$
P(\mathcal{D},\ph,\epsilon; f) = \varlimsup_{n\to\infty} \frac 1n \log \Lambda^{\mathrm{sep}}_n(\mathcal{D},\ph,\epsilon),
$$
and the \emph{pressure of $\varphi$ on $\mathcal{D}$}
$$
P(\mathcal{D},\ph; f) = \lim_{\epsilon\to 0}P(\mathcal{D},\ph,\epsilon;f).
$$
The above definition 
is a non-stationary version of the usual upper capacity pressure \cite{Pesin}.  For a set $Z \subset X$, we let $P(Z, \varphi, \epsilon; f) := P(Z \times \NN, \varphi, \epsilon; f)$, and thus $P(Z, \varphi; f)$ is the usual upper capacity pressure.

When $\ph=0$ the above definition gives the \emph{entropy of $\mathcal{D}$}:
\begin{equation}\label{eqn:h}
\begin{aligned}
h(\mathcal{D}, \epsilon; f)= h(\mathcal{D}, \epsilon) &:= P(\mathcal{D}, 0, \epsilon) \mbox{ and } h(\mathcal{D})= \lim_{\epsilon\rightarrow 0} h(\mathcal{D}, \epsilon).
\end{aligned}
\end{equation}

We let $\mathcal{M}(f)$ denote the set of $f$-invariant Borel probability measures and $\mathcal{M}_e(f)$ the set of ergodic $f$-invariant Borel probability measures.
The variational principal for pressure \cite[Theorem 9.10]{Wal} states that if $X$ is a compact metric space and $f$ is continuous, then 
\[
P(\varphi; f)=\sup_{\mu\in \mathcal{M}(f)}\left\{ h_{\mu}(f) +\int \varphi \,d\mu\right\}
=\sup_{\mu\in \mathcal{M}_e(f)}\left\{ h_{\mu}(f) +\int \varphi \,d\mu\right\}.
\]
A measure achieving the supremum is an \emph{equilibrium state}, and these are the objects whose existence and uniqueness we wish to study.

\subsection{Expansivity and tail entropy}\label{sec:tail-entropy}

Given a homeomorphism $f\colon X\to X$ and $\eps>0$, consider for each $x\in X$ and $\eps>0$ the set
\[
\Gamma_\eps(x) := \{y\in X \, :\,  d(f^kx,f^ky) < \eps \text{ for all } n\in \ZZ \}
\]
is the \emph{(bi-infinite) Bowen ball of $x$ of size $\epsilon$}.
Note that $f$ is expansive if and only if there exists $\eps>0$ so that $\Gamma_\eps(x)= \{x\}$ for all $x\in X$.  

For systems that fail to be expansive, it is useful to consider the \emph{tail entropy} of $f$ at scale $\eps>0$ is
\begin{equation}\label{eqn:Bowen-tail-entropy}
h_f^*(\epsilon) = \sup_{x\in X} \lim_{\delta\to 0} \limsup_{n\to\infty}
\frac 1n \log \Lspan_n(\Gamma_\eps(x)\times \mathbb{N},0,\delta;f).
\end{equation}
This quantity was introduced in \cite{rB72}; equivalent definitions can also be formulated using open covers \cite{mM76}.

The map $f$ is entropy-expansive if $h_f^*(\eps)=0$ for some $\eps>0$, and is asymptotically $h$-expansive if $h_f^*(\eps)\to 0$ as $\eps\to 0$. See \cite{BD, dB09} for connections between these notions and the theory of symbolic extensions. An interesting result of \cite{BD} is that positive tail entropy rules out the existence of a principal symbolic extension, and thus symbolic dynamics fails in a strong way for such systems.

Given a collection $\DDD\subset X\times \NN$ and scales $0<\delta<\eps$, the tail entropy allows us to control
pressure at scale $\delta$ in terms of pressure at scale $\eps$. 
The following is proved in \S\ref{s.lemmas}.
  
\begin{lem} \label{smallerscales}
Given any $\DDD\subset X\times \NN$ and $0<\delta<\eps$, we have
\[
P(\DDD,\ph,\delta;f) \leq P(\DDD,\ph,\epsilon; f) + h_f^\ast(\epsilon) + \Var(\ph,\epsilon) + \Var(\ph,\delta).
\]
In particular, $P(\DDD,\ph; f) \leq P(\DDD,\ph,\epsilon; f) + h_f^\ast(\epsilon) + \Var(\ph,\epsilon)$.
\end{lem}


\subsection{Obstructions to expansivity, specification, and regularity}

It was shown by Bowen \cite{Bow75} that $(X,f,\ph)$ has a unique equilibrium state whenever $(X,f)$ has expansivity and specification, and $\ph$ is sufficiently regular.  We require the results from \cite{CT4}, which give existence and uniqueness as long as  `obstructions to specification and regularity' and `obstructions to expansivity' have smaller pressure than the whole system. We recall the necessary definitions, which can be found in \cite{CT4}.



\subsubsection{Expansivity}

We introduce the following quantity associated with the set of  non-expansive points. 
\begin{defn} \label{almostexpansive}
For $f\colon X\rightarrow X$ the set of non-expansive points at scale $\epsilon$ is  $\mathrm{NE}(\epsilon):=\{ x\in X \, :\,  \Gamma_\eps(x)\neq \{x\}\}$.  An $f$-invariant measure $\mu$ is  almost expansive at scale $\epsilon$ if $\mu(\mathrm{NE}(\epsilon))=0$.  Given a potential $\varphi$, the pressure of obstructions to expansivity at scale $\epsilon$ is
\begin{align*}
\Pexp(\varphi, \epsilon) &=\sup_{\mu\in \mathcal{M}_e(f)}\left\{ h_{\mu}(f) + \int \varphi\, d\mu\, :\, \mu(\mathrm{NE}(\epsilon))>0\right\} \\
&=\sup_{\mu\in \mathcal{M}_e(f)}\left\{ h_{\mu}(f) + \int \varphi\, d\mu\, :\, \mu(\mathrm{NE}(\epsilon))=1\right\}.
\end{align*}
This is monotonic in $\eps$, so we can define a scale-free quantity by
\[
\Pexp(\varphi) = \lim_{\epsilon \to 0} \Pexp(\varphi, \epsilon).
\]
\end{defn}


\subsubsection{Specification}

We define specification for a collection of orbit segments.
\begin{defn} 
A collection of orbit segments $\mathcal{G}\subset X\times \mathbb{N}$ has \emph{specification at scale $\epsilon$} if there exists $\tau\in\mathbb{N}$ such that for every $\{(x_j, n_j)\, :\, 1\leq j\leq k\}\subset \mathcal{G}$,
there is a point $x$ in
$$\bigcap_{j=1}^k f^{-(m_{j-1}+ \tau)}B_{n_j}(x_j, \epsilon),$$
where $m_{0}=-\tau$ and $m_j = \left(\sum_{i=1}^{j} n_i\right) +(j-1)\tau$ for each $j \geq 1$.
\end{defn}

The above definition says that there is some point $x$ whose trajectory shadows each of the $(x_i,n_i)$ in turn, taking a transition time of exactly $\tau$ iterates between each one.  The numbers $m_j$ for $j\geq 1$ are the time taken for $x$ to shadow $(x_1, n_1)$ up to $(x_j, n_j)$.

It is sometimes convenient to consider collections $\GGG$ in which only long orbit segments have specification.
\begin{defn}
A collection of orbit segments $\mathcal{G}\subset X\times \mathbb{N}$ has \emph{tail specification at scale $\epsilon$} if there exists $N_0\in \NN$ so that the collection $\GGG_{\geq N_0} :=\{(x,n) \in \GGG \mid n \geq N_0\}$ has specification at scale $\epsilon$.
\end{defn}


\subsubsection{Regularity}

We require the following regularity condition for the potential $\ph$ on the collection $\GGG$. 

\begin{defn} \label{Bowen}
Given $\mathcal{G}\subset X\times \mathbb{N}$, a potential $\varphi$  has the \emph{Bowen property on $\GGG$ at scale $\epsilon$} if
\[
V(\GGG,\ph,\epsilon) := \sup \{ |S_n\varphi (x) - S_n\varphi(y)| : (x,n) \in \GGG, y \in B_n(x, \epsilon) \} <\infty.
\]
We say $\varphi$ has the \emph{Bowen property on $\GGG$} if there exists $\epsilon>0$ so that $\varphi$ has the Bowen property on $\GGG$ at scale $\epsilon$.
\end{defn}

If $\GGG$ has the Bowen property at scale $\eps$, it has it for all smaller scales. 


\subsection{General results on uniqueness of equilibrium states}\label{sec:CT}

We prove existence and uniqueness of equilibrium states by using Theorem 5.6 of \cite{CT4}.  The idea is to find a collection of orbit segments $\GGG\subset X\times \mathbb{N}$ that satisfies specification and the Bowen property, and that is sufficiently large in an appropriate sense.  To make this precise, we need the following definition.  We denote $\mathbb{N}_0=\mathbb{N}\cup\{0\}$.

\begin{defn}
A \emph{decomposition} for $(X,f)$ consists of three collections $\mathcal{P}, \mathcal{G}, \mathcal{S}\subset X\times \mathbb{N}_0$ and three functions $p,g,s\colon X\times \mathbb{N}\to \NN_0$ such that for every $(x,n)\in X\times \NN$, the values $p=p(x,n)$, $g=g(x,n)$, and $s=s(x,n)$ satisfy $n = p+g+s$, and 
\begin{equation}\label{eqn:decomposition}
(x,p)\in \mathcal{P}, \quad (f^p(x), g)\in\mathcal{G}, \quad (f^{p+g}(x), s)\in \mathcal{S}.
\end{equation}
Given a decomposition $(\PPP,\GGG,\SSS)$ and $M\in \mathbb{N}$, we write $\GGG^M$ for the set of orbit segments $(x,n)$ for which 
$p \leq M$ and $s\leq M$.
\end{defn}
  
Note that the symbol $(x,0)$ denotes the empty set, and the functions $p, g, s$ are permitted to take the value zero.

\begin{thm}[Theorem 5.6 of \cite{CT4}] \label{t.generalBV}
Let $X$ be a compact metric space and $f\colon X\to X$ a homeomorphism. Let $\ph \colon X\to\RR$ be a continuous potential function.  Suppose there exists $\epsilon>0$ such that $\Pexp(\ph, \twentyone\epsilon) < P(\ph)$ and $X\times \NN$ admits a decomposition $(\PPP, \GGG, \SSS)$ with the following properties:
\begin{enumerate}[label=\upshape{(\arabic{*})}]
\item\label{c.spec} For each $M \geq 0$, $\GGG^M$ has tail specification at scale $\epsilon$;
\item\label{c.bowen} $\ph$ has the Bowen property at scale $\twentyone \epsilon$ on $\GGG$;
\item\label{c.gap}
 $P(\PPP \cup \SSS,\ph, \epsilon) + \Var(\ph, \twentyone \epsilon) < P(\ph)$.
\end{enumerate}
Then there is a unique equilibrium state for $\ph$.
\end{thm}
We comment on these hypotheses. 
 The transition time $\tau$ for specification for $\GGG^M$ depends on $M$.  If $\GGG$ had specification at all scales, then a simple argument \cite[Lemma 2.10]{CT4} based on modulus of continuity of $f$ shows that the first hypothesis of the theorem is true for any $\epsilon$. Thus, considering $\GGG^M$ for all $M$ at a fixed scale stands in for controlling $\GGG$ at all scales. The Bonatti--Viana example is a situation where 
we can verify specification at a fixed scale for every $\GGG^M$, even though we do not expect $\GGG$ to have specification at all scales. The reason for this is that our proof of specification at scale $\eps$ requires us to start with a piece of a centre-unstable leaf of size $\eps$ and iterate it until it becomes $\eps$-dense in $\TT^4$; see \S\ref{sec:spec}.  As long as $\eps$ is larger than the size of the perturbation, this happens in a uniformly bounded number of iterates, independent of which centre-unstable leaf we start with; see Lemmas \ref{lem:large-scale} and \ref{lem:overflowing}.  When $\eps$ is smaller than the scale of the perturbation, though, the number of iterates it takes for this leaf to become $\eps$-dense need not be uniformly bounded, and so our proof of specification no longer works at these small scales.

 There are two scales present in Theorem \ref{t.generalBV}: $\eps$ and $\twentyone \eps$.
We require specification at scale $\eps$, while expansivity and the Bowen property are controlled at the larger scale $\twentyone \eps$. There is nothing fundamental about the constant $\twentyone$, but it is essential that expansivity and the Bowen property are controlled at a larger scale than specification. This is because every time we use specification in our argument to estimate an orbit, we move distance up to $\epsilon$ away from our original orbit, and we need to control expansivity and regularity properties for orbits after multiple applications of the specification property. The  $\Var(\ph, \twentyone \epsilon)$ term appears because we must control points that are distance up to $\twentyone \eps$ from a separated set for $\PPP \cup \SSS$. 


\section{Perturbations of Anosov Diffeomorphisms}\label{s.perturb}
In this section, we collect some more background material about weak forms of hyperbolicity, and perturbations of Anosov diffeomorphisms. We also establish 
a pressure estimate for $C^0$ perturbations of Anosov diffeomorphisms that plays a key role in 
our results.
\subsection{Dominated splittings} \label{basics}


Let $M$ be a compact manifold and let  $f\colon M\to M$ be a diffeomorphism.
A $Df$-invariant vector bundle $E \subseteq TM$ has a {\it dominated splitting} if
$$
E = E_1\oplus\cdots\oplus E_k,
$$
where each subbundle $E_i$ is $Df$-invariant with constant dimension,  
and there exists an integer $\ell\geq1$ with the following property: for every $x\in M$, all $i=1,\dots,(k-1)$,  and every pair of unit vectors $u\in E_1(x)\oplus\dots\oplus E_i(x)$ and $v\in E_{i+1}(x)\oplus\dots\oplus E_k(x)$,  it holds that
$$
 \frac{|Df_x^\ell (u)|}{|Df^{\ell}_x (v)|}\leq \frac{1}{2}.
$$
See for example \cite{mS16} or ~\cite[Appendix B, Section 1]{BDV} for some properties of systems with a dominated splitting.


For us, $k=2$, and we obtain a dominated splitting $TM = E^{cs} \oplus E^{cu}$, and there exist invariant foliations $W^{cs}$ and $W^{cu}$ tangent to $E^{cs}$ and $E^{cu}$ respectively that we call the \emph{centre-stable and centre-unstable foliations}.  For $x\in M$ we let $W^{\sigma}(x)$ be the leaf of the foliation $\sigma\in \{cs, cu\}$ containing $x$.
Given $x\in M$ and $y\in W^\sigma(x)$, we write
\begin{multline}\label{eqn:dW}
d_{W^\sigma}(x,y) = \inf\{ \mathrm{length}(\gamma) : \gamma\text{ is a path from $x$ to $y$}\\ \text{that is wholly contained in $W^\sigma(x)$} \}.
\end{multline}
Given $\eta>0$, we write $W^\sigma_\eta(x) = \{y\in W^\sigma(x) : d_{W^\sigma}(x,y) \leq \eta\}$.

Suppose $W^1,W^2$ are foliations of $M$. 
The standard notion of local product structure for $W^1,W^2$ says that for every $x,y\in M$ that are close enough to each other, the local leaves $W^u_\mathrm{loc}(x)$ and $W^s_\mathrm{loc}(y)$ intersect in exactly one point. Our definition of local product structure additionally keeps track of the scales involved. 
We say that $W^1,W^2$ have \emph{local product structure at scale $\eta>0$ with constant $\kappa\geq 1$} if for every $x,y\in M$ with $\eps := d(x,y) < \eta$, the leaves $W^1_{\kappa \eps}(x)$ and $W^2_{\kappa \eps}(y)$ intersect in a single point.

\subsection{Constants associated to Anosov maps} \label{constants}
Let $f\colon M\to M$ be a transitive Anosov diffeomorphism. A constant that will be important for us is the constant $C=C(f)$ arising from the Anosov shadowing lemma \cite{KH},  \cite[Theorem 1.2.3]{sP99}.

\begin{lem}[Anosov Shadowing Lemma]\label{shadowinglemma} 
Let $f$ be a transitive Anosov diffeomorphism. There exists $C=C(f)$ so that if $2 \eta>0$ is an expansivity constant for $f$, then every $\frac\eta C$-pseudo-orbit $\{x_n\}$ for $f$ can be $\eta$-shadowed by an orbit $\{y_n\}$ for $f$.
\end{lem}

Another constant that will appear in our analysis is $L=L(f, \eta)$ associated with the growth of certain partition sums for $f$. Recall that $f$ is expansive and has the specification property, and let $h=\htop(f)$ be the topological entropy . For any $\eta>0$ smaller than the expansivity constant for $f$,  there is a constant $L =  L(f, \eta)$ so that
\begin{equation}\label{eqn:Lambda-upper-bound}
\Lambda^{\mathrm{sep}}_n(M\times\mathbb{N},0,\eta; f) \leq  L e^{nh}
\end{equation}
for every $n$. See, e.g. \cite[Lemma 3]{Bow75}. 
The constant $L$ can be determined explicitly in terms of the transition time in the specification property.

\subsection{Partition sums for $C^0$ perturbations} 
\label{sec:perturbation-sums}

Let $f\colon M\to M$ be a transitive Anosov diffeomorphism of a compact manifold.  Using the Anosov shadowing lemma, we show that there is a $C^0$-neighborhood $\mathcal{U}$ of $f$ such that for every $g\in \mathcal{U}$, there is a natural map from $g$ to $f$ given by sending a point $x$ to a point whose $f$-orbit shadows the $g$-orbit of $x$. 
It is a folklore result that this map is a semi-conjugacy when $\mathcal{U}$ is sufficiently small.\
{For example, this follows from the proof of \cite[Proposition 4.1]{BFSV}}. This allows us to control partition sums of $g$ at large enough scales from above, and the pressure at all scales from below; {the following lemma is proved in \S \ref{s.lemmas}.}

\begin{prop} \label{pressuredrop}
Let $f$ be a transitive  Anosov diffeomorphism. Let $C= C(f)$ be the constant from the Anosov shadowing lemma, and $3 \eta>0$ be an expansivity constant for $f$.  If $g\in \Diff(M)$ is such that $d_{C^0}(f,g) < \eta/C$, then:
\begin{enumerate}[label=\upshape{(\roman{*})}]
\item\label{Pg-geq}
 $P(\varphi; g)\geq P(\varphi; f)- \Var(\varphi, \eta)$;
\item\label{Lambdag-leq}
 $\Lambda^{\mathrm{sep}}_n(\varphi, 3 \eta ;g) \leq \Lambda^{\mathrm{sep}}_n(\varphi, \eta ;f)e^{n \Var(\varphi, \eta) }$.
\end{enumerate}
\end{prop}

It follows from \ref{Lambdag-leq} that
\begin{equation}\label{eqn:Pg-leq}
P(\varphi, 3\eta; g) \leq P(\varphi; f) + \Var(\varphi, \eta).
\end{equation}
However, it may be that $P(\ph; g)$ is greater than $P(\ph,3\eta; g)$ due to the appearance of entropy at smaller scales for $g$.  Nonetheless, we can obtain an upper bound on $P(\ph; g)$ which involves the tail entropy; Lemma \ref{smallerscales} 
and \eqref{eqn:Pg-leq} together give the bound
\begin{equation}\label{eqn:Pgphleq}
P(\ph; g) \leq P(\ph; f) + h_g^*(3\eta) + 2\Var(\ph,3\eta).
\end{equation}
The pressure of $g$, and consequently the tail entropy term, can be arbitrarily large for a $C^0$ perturbation of $f$.  For example, $f$ can be perturbed continuously in a neighborhood of a fixed point to create a whole disc of fixed points, and then composed with a homeomorphism of this disc that has arbitrarily large entropy.

\subsection{Pressure estimates} \label{pressureestimate}

The examples that we consider are obtained as $C^0$-perturbations of Anosov maps, where the perturbation is made inside a small neighborhood of some fixed points.  Our strategy is to apply the abstract uniqueness results of Theorem \ref{t.generalBV} when $\PPP,\SSS$ are orbit segments spending a large proportion of their time near the fixed points. 
In this section we give an estimate on the pressure carried by such orbit segments.  We fix the following data.
\begin{itemize}
\item Let $f\colon M\to M$ be a transitive Anosov diffeomorphism of a compact manifold, with topological entropy $h=\htop(f)$.
\item Let $q$ be a fixed point for $f$.
\item Let $3 \eta$ be an expansivity constant for $f$. 
\item Let $C=C(f)$ be the constant from the shadowing lemma.
\item Let $L =  L(f, \eta)$ be a constant so that \eqref{eqn:Lambda-upper-bound} holds.
\end{itemize}
Now we choose $g$, $\CCC$, and $\ph$:
\begin{itemize}
\item Let $g\colon M\to M$ be a diffeomorphism with $d_{C^0}(f, g)< \eta/C$.
\item Let $\rho<3\eta$.
\item Let $r >0$ be small, and let $\CCC=\CCC(q,r;g)$ be the collection of all orbit segments $(x,n)$ that spend less than $r$ of their time outside $B(q,\rho)$, that is,
\begin{equation}\label{eqn:C}
\CCC =\CCC(q,r;g) = \{(x,n)\in M\times \mathbb{N} \, :\,  S_n^g\chi_q(x)< n r \},
\end{equation}
where $\chi_q$ is the indicator function of $M \setminus B(q, \rho)$.
 \item Let $\ph$ be any continuous function.
\end{itemize}
We write $H(r) = -r\log r - (1-r) \log(1-r)$. We have the following entropy and pressure estimates on $\CCC$.

\begin{thm} \label{coreestimateallscales}
Under the assumptions above, we have the inequality 
\begin{equation} \label{eqn:hC}
h(\CCC,6\eta;g) \leq r(\htop(f) + \log L) +H(2r),
\end{equation}
and the inequality that for any scale $\delta>0$,
\begin{equation} \label{eqn:PC}
P(\CCC,\ph,\delta;g) \leq (1-r) \sup_{x\in B(q, \rho)}\ph(x) + r \sup_{x\in M}\ph(x) + h(\CCC,\delta;g),
\end{equation}
and thus it follows that
\[
P(\CCC, \ph; g) \leq  h_g^\ast(6 \eta) +(1-r) \sup_{x\in B(q, \rho)}\ph(x) + r( \sup_{x\in M}\ph(x) + h + \log L ) + H(2r).
\]
\end{thm}
\begin{proof} 
First we prove the entropy estimate \eqref{eqn:hC}. For each $(x,n)\in \CCC$, we partition its orbit into segments entirely in $B(q, \rho)$, and segments entirely outside $B(q, \rho)$.
More precisely, given $(x,n)\in \CCC$, let $((x_i,n_i), (y_i,m_i))_{i=1}^{\ell}$ 
be the uniquely determined sequence such that
\begin{itemize}
\item {$x_0=x$ and $\sum_{i=1}^\ell (n_i + m_i) = n$;}
\item $g^{n_i}(x_i) = y_i$ and $g^{m_i}(y_i) = x_{i+1}$;
\item $x_i \in B_{n_i}(q, \rho)$ (letting $n_0=0$ if $x \notin B(q, \rho)$); 
\item $(y_i,m_i)$ corresponds to an orbit segment entirely contained in $M \setminus B(q, \rho)$ (letting $m_\ell =0$ if $g^{n-1}x \in B(q, \rho)$).
\end{itemize}
Note that $\ell = \ell(x,n)$ satisfies $\ell-1 \leq \sum_{i=1}^\ell m_i = S_n^g\chi (x) <  n r$. 
For $(x,n) \in \CCC$, let
\[
\underline t (x,n) = 
{(\ell, \mm, \nn) = (\ell, (m_1,\dots, m_\ell), (n_1,\dots, n_\ell))}
\]
be the time data obtained this way.
{Given $n\in \NN$ and $r>0$, let
\[
\JJJ_n^r = \{(\ell,\mm,\nn) : 1\leq \ell \leq nr+1,\
\textstyle \sum(m_i + n_i) = n,\ \sum m_i < nr \}.
\]
Writing $(\CCC_n)_{\ell,\mm,\nn} = \{(x,n)\in \CCC_n : \underline{t}(x,n) = (\ell,\mm,\nn)\}$, we have
\[
\CCC_n = \bigcup_{(\ell,\mm,\nn) \in \JJJ_n^r} (\CCC_n)_{\ell,\mm,\nn}.
\]
Thus we can estimate $\Lsep_n(\CCC_n,0,6\eta)$ in terms of $\Lsep_n((\CCC_n)_{\ell,\mm,\nn},0,6\eta)$ and $\#\JJJ_n^r$.
}

{For the first of these,} let $E_n\subset \CCC_n$ be $(n,6\eta)$-separated, and let $F_n$ be {maximally} $(n, 3\eta)$-separated, and thus $(n, 3\eta)$-spanning, for $M$. Note that if $z_1, z_2 \in (\CCC_n)_{\ell, \mm, \nn}$, then $d_{n_i}(g^{s_{i-1}} z_1, g^{s_{i-1}} z_2)<2 \rho<6\eta$ at times $s_i$ which correspond to the orbits entering $B_{n_i}(q, \rho)$; that is, for $s_0 =0$ and $s_{i-1} = \sum_{j=1}^{i-1} (n_j+m_j)$. Thus, if $z_1, z_2 \in E_n \cap (\CCC_n)_{\ell, \mm, \nn}$ with $z_1 \neq z_2$, then there exists $i$ with $d(g^iz_1, g^iz_2)>6\eta$, and the time $i$ can occur only when the orbit segments are outside $B(q, \rho)$. More precisely, let $r_0=n_1$, $r_1=n_1+m_1+n_2$, and $r_i= \sum_{j=1}^{i+1}n_i + \sum_{j=1}^{i}m_i$.  There must exist $i$ so that $d_{m_i}(g^{r_{i-1}}z_1, g^{r_{i-1}}z_2) > 6 \eta$.

We define a map $\pi\colon (\CCC_n)_{\ell, \mm, \nn} \to F_{m_1} \times \cdots \times F_{m_\ell}$
by choosing $\pi_i(z) \in F_{m_i}$ with the property that $d_{m_i}(g^{r_{i-1}}z, \pi_i(z)) \leq 3 \eta$. It follows from the above that if  $z_1, z_2 \in E_n \cap (\CCC_n)_{\ell, \mm, \nn}$ with $z_1 \neq z_2$, there exists $i$ with $d_{m_i}(g^{r_{i-1}}z_1, g^{r_{i-1}}z_2) > 6 \eta$, and thus $\pi_i(z_1) \neq \pi_i(z_2)$. Thus, the map $\pi$ is injective.

Recall that $L$ is the constant such that \eqref{eqn:Lambda-upper-bound} holds {and that $h=\htop(f)$}.  Since $d_{C_0}(f, g) < \eta/C$
, using Proposition \ref{pressuredrop}, we have
\begin{equation}\label{eqn:upper-bound}
\Lsep_m(M,0,3\eta; g) \leq  \Lsep_m(M,0,\eta;f) \leq L e^{mh}.
\end{equation}
Thus it follows from injectivity of the map $\pi$ that
\[
\Lsep_n((\CCC_n)_{\ell, \mm, \nn},0,6\eta) \leq \prod_{i=1}^\ell  \Lsep_{m_i}(M,0,3\eta;g) \leq L^{\ell} e^{(\sum m_i) h} \leq L^{nr+1} e^{nrh},
\]
and thus summing over all choices of $\ell, \mm,\nn$, we obtain
\[
\Lsep_n(\CCC_n,0,6\eta) \leq \sum_{(\ell,\mm,\nn)\in \JJJ_n^r} \Lsep_n((\CCC_n)_{\ell, \mm, \nn},0,6\eta)  \leq L^{nr+1} (\# 
\JJJ_n^r) e^{nrh}.
\]
Now we observe that {given $1\leq \ell\leq nr+1$,} 
the choice of $\mm,\nn$ is uniquely determined by choosing 
$2 \ell-1$ 
 elements of $\{0,1,\dots,n-1\}$, which are the partial sums of $m_i$ and $n_i$ (the times when the trajectory enters or leaves $B(q,\rho)$, {denoted by $r_i$ and $s_i$ above}).  
An elementary computation using Stirling's formula or following \cite[Lemma 5.8]{C15} shows that the  number of such $\ell, \mm,\nn$ is at most 
\[
\sum_{k=1}^{2nr+1} 
\binom{n}{k}
\leq 
(2nr+1) (n+1) e^{n H((2nr+1)/n) + 1},
\]
and so we have
\[
\Lsep_n(\CCC,0,6\eta) \leq L^{nr+1} (2nr+1)(n+1) e^{ nrh} e^{nH(2r+\frac 1n)}.
\]
This gives the bound $h(\CCC,6\eta;g) \leq r(\htop(f) + \log L)  +H(2r)$,
which establishes \eqref{eqn:hC}. The pressure estimate \eqref{eqn:PC} follows from \eqref{eqn:hC} by observing that for every $(x,n)\in \CCC$ we have $g^kx\in \overline B(q,\rho)$ for at least $(1-r)n$ values of $k\in \{0,1,\dots,n-1\}$, and so
\[
S_n^g\ph(x) \leq (1-r)n \sup_{x\in B(q,\rho)} \ph(x) + rn \sup_{x\in M} \ph(x);
\]
this yields the partition sum estimate
\[
\Lsep_n(\CCC, \ph,\delta;g) \leq \Lsep_n(\CCC, 0,\delta;g) \exp (n \{(1-r)\sup_{x\in B(q, \rho)}\ph(x) + r \sup_{x\in M}\ph(x)\}),
\]
which implies \eqref{eqn:PC}.  The third displayed inequality of Theorem \ref{coreestimateallscales} is immediate from the inequalities \eqref{eqn:hC}, \eqref{eqn:PC} and Lemma \ref{smallerscales}.
\end{proof}



\subsection{Obstructions to expansivity}\label{sec:obstr-exp} 

The diffeomorphisms $g$ that we consider will be shown to satisfy the following expansivity property, where we continue to write $\chi_q$ for the indicator function of $M\setminus B(q,\rho)$:
\begin{enumerate}[label=\upshape{\textbf{[\Alph{*}{]}}}]
\setcounter{enumi}{4}
\item\label{E} there exist $\epsilon >0$, $r>0$, and fixed points $q,q'$ 
such that 
for $x \in M$, if there exists a sequence $n_k\to\infty$ with $\frac{1}{n_k}S^g_{n_k}\chi_q(x) \geq r$, and a sequence $m_k\to\infty$ with $\frac{1}{m_k}S^{g^{-1}}_{m_k}\chi_{q'}(x) \geq r$, then $\Gamma_\eps(x)=\{x\}$.
\end{enumerate}

In the previous section, and this one, formally $q$ and $q'$  could be any fixed points for $g$ that verify condition \ref{E}.  In applying this to our main results, we naturally take $q,q'$ to be the points around which we make the perturbation that defines the Bonatti--Viana examples.

\begin{thm} \label{expansivityestimate}
If $g$ is as in the previous section and $q,q'$ are such that \ref{E} holds, then we have 
$\Pexp(\ph,\eps) \leq P(\CCC(q,r) \cup \CCC(q',r),\ph)$.
\end{thm}

\begin{proof} 
Write $\chi=\chi_q, \chi'=\chi_{q'}$, $\CCC=\CCC(q,r;g)$, $\CCC'=\CCC(q',r;g)$.  Consider the sets
\begin{equation}\label{eqn:A+}
\begin{aligned}
A^+ &= \{x : \text{there is } K(x) \text{ so } \tfrac{1}{n}S^g_{n}\chi(x) < r \text{ for all } n>K(x)\}, \\
A^- &= \{x : \text{there is } K(x) \text{ so } \tfrac{1}{n}S^{g^{-1}}_{n}\chi'(x) < r \text{ for all } n>K(x)\}.
\end{aligned}
\end{equation}
Theorem \ref{expansivityestimate} is an application of the following theorem, whose proof is based on the Katok pressure formula \cite{lM88}.

\begin{lem} \label{keystepexpansivityestimate}
Let $\mu \in \mathcal{M}_e(g)$. If either $\mu(A^+) >0$ or $\mu(A^-)>0$, then 
$h_{\mu}(g)+\int \ph\, d \mu \leq P(\CCC\cup\CCC', \ph)$.
\end{lem}
\begin{proof} 
Start with the case where $\mu(A^+)>0$; we show that $h_\mu(g) + \int\ph\,d\mu \leq P(\CCC,\ph)$.  Given $k\in \NN$, let $A_k^+ = \{ x \in A^+ \, :\,  K(x) \leq k\}$, and observe that $\mu (\bigcup_k A_k^+) > 0$, so there is some $k$ such that $\mu(A_k^+)>0$.

Note that for every $n>k$ and $x\in A_k^+$, we have $(x,n) \in \CCC$.  It follows that for every $\delta>0$ we have
\begin{equation}\label{eqn:PAk}
\Lsep_n(A_k^+, \ph,\delta; g) \leq \Lsep_n(\CCC,\ph,\delta;g).
\end{equation}
Fix $\alpha \in (0,\mu(A_k^+))$ and consider the quantity
\[
s_n(\ph, \delta, \mu,\alpha ; g)=\mathrm{inf}\left \{ \sum_{x\in E} \exp\{S^g_n\ph(x)\} :  \mu \left(\bigcup_{x\in E} \overline B_n(x, \delta) \right)\geq \alpha \right \},
\]
where the infimum is taken over finite subsets $E\subset X$.  
The pressure version of the Katok entropy formula  \cite{lM88} states that
\[
h_{\mu}(g) + \int \ph\, d \mu=\lim_{\delta\rightarrow 0}\limsup_{n\rightarrow \infty}\frac{1}{n}\log s_n(\ph, \delta, \mu, \alpha; g).
\]
Note that $s_n(\ph, \delta, \mu, \alpha; g) \leq \Lspan_n(A_k^+, \ph, \delta;g) \leq \Lsep_n(A_k^+, \ph, \delta;g) \leq \Lsep_n(\CCC,\ph,\delta;g)$.
It follows that 
\[
h_{\mu}(g)+\int \ph\, d \mu \leq P(\CCC, \ph) = \lim_{\delta \to 0} P(\CCC, \ph,  \delta).
\]
The case where $\mu(A^-)>0$ is similar: obtain $A_k^-\subset A^-$ such that $K(x)\leq k$ for all $x\in A_k^-$ and $\mu(A_k^-)>0$.  Then  observe that for $x \in A_k^-$, we have $(g^{-n}x, n) \in \CCC'$ for any $n \geq k$.  Moreover, $(n, \epsilon)$-separated sets for $g$ are in one to one correspondence with $(n, \epsilon)$-separated sets for $g^{-1}$, and $S_n^{g^{-1}}\ph (x) = S_n^g \ph(g^{-n+1}x)$.  Then a simple argument shows that $P(A_k^-, \ph, \epsilon; g^{-1}) \leq P(\CCC',  \ph, \epsilon;g)$. 

Finally, Katok's pressure formula applied to $g^{-1}$ tells us that
\[
h_{\mu}(g) + \int \ph\, d \mu=\lim_{\delta\rightarrow 0}\limsup_{n\rightarrow \infty}\frac{1}{n}\log s_n(\ph, \delta, \mu, \alpha; g^{-1}).
\]
Thus $h_{\mu}(g)+\int \ph\, d \mu \leq \lim_{\delta \to 0}  P(A_k^-, \ph, \epsilon; g^{-1}) \leq\lim_{\delta \to 0} P(\CCC', \ph,  \delta)$. 
\end{proof}

Now, to prove Theorem \ref{expansivityestimate}, by the hypothesis \ref{E}, if $\Gamma_\eps(x) \neq \{x\}$, then either there are only finitely many $n$ so that $\frac{1}{n}S^g_{n}\chi(x) \geq r$, or there are only finitely many $n$ so that $\frac{1}{n}S^{g^{-1}}_{n}\chi'(x) \geq r$. Thus, if $x\in \mathrm{NE}(\epsilon)$, then either $x\in A^+$ or $x \in A^-$. Thus, if $\mu$ is an ergodic measure satisfying $\mu(\mathrm{NE}(\epsilon)) >0$; then at least one of $A^+$ or $A^-$ has positive $\mu$-measure.  Thus, Theorem \ref{keystepexpansivityestimate} applies, and we conclude that 
\[
h_{\mu}(g)+\int \ph\, d \mu \leq P(\CCC \cup \CCC', \ph). \qedhere 
\]
\end{proof}


\subsection{Cone estimates and local product structure}\label{sec:torus}

Let $F^1,F^2 \subset \RR^d$ be subspaces such that $F^1 \cap F^2 = \{0\}$ (we do not assume that $F^1 + F^2 = \RR^d$).  Let
$\measuredangle(F^1,F^2) := \min\{\measuredangle(v,w) \, :\,  v\in F^1, w\in F^2\}$, and consider the quantity $\bk(F^1,F^2) :=  (\sin\measuredangle(F^1,F^2))^{-1} \geq 1$.  Some elementary trigonometry shows that 
\begin{equation}\label{eqn:barkappa}
\|v\| \leq \bk(F^1,F^2) \text{ for every } v\in F^1 \text{ with } d(v,F^2) \leq 1,
\end{equation}
or equivalently,
\begin{equation}\label{eqn:barkappa2}
\|v\| \leq \bk(F^1,F^2) d(v,F^2) \text{ for every } v\in F^1.
\end{equation}
Given $\beta \in (0,1)$ and $F^1,F^2 \subset \RR^d$, the \emph{$\beta$-cone of $F^1$ and $F^2$} is 
\[
C_\beta(F^1,F^2) = \{ v+w \, :\,  v\in F^1, w\in F^2, \|w\| < \beta \|v\| \}.
\]

\begin{lem}\label{lem:Wlps}
Let $W^1,W^2$ be any foliations of $F^1 \oplus F^2$ with $C^1$ leaves such that $T_x W^1(x) \subset C_\beta(F^1,F^2)$ and $T_x W^2(x) \subset C_\beta(F^2,F^1)$, and let $\bk = \bk(F^1,F^2)$.  Then for every $x,y\in F^1 \oplus F^2$ the intersection $W^1(x) \cap W^2(y)$ consists of a single point $z$.  Moreover,
\[
\max\{d_{W^1}(x,z), d_{W^2}(y,z)\} \leq \frac{1+\beta}{1-\beta} \bk d(x,y),
\]
where $d_{W^i}$ is as in \eqref{eqn:dW}.
\end{lem}

We will consider foliations on $\TT^4$ whose lifts to $\RR^4$ satisfy the hypotheses of Lemma \ref{lem:Wlps}.  Uniqueness of the intersection point on $\TT^4$ follows from restricting to sufficiently small local leaves. We also need the following lemma, 
which compares the intrinsic distance along a leaf with the distance induced from the metric on $\TT^4$.
\begin{lem}\label{compare:dist}
Under the assumptions of Lemma \ref{lem:Wlps}, suppose that $x, y$ are points belonging to the same local leaf of $W\in \{W^1, W^2\}$. Then
\[
d(x,y) \leq d_W(x,y) \leq (1+\beta)^2 d(x,y).
\]
\end{lem}


Lemmas \ref{lem:Wlps} and \ref{compare:dist} are proved in \S\ref{s.lemmas}.


\section{Bonatti-Viana construction and Main result}\label{sec:BV}
In~\cite{BV}, Bonatti and Viana defined a $C^1$-open class of diffeomorphisms by a list of 4 hypotheses, which ensure robust transitivity and the existence of a dominated splitting into two bundles with no invariant sub-bundles. They then gave an explicit construction of a family of diffeomorphisms satisfying these 4 hypotheses, thus showing that the Bonatti-Viana class is non-empty. We refer to this as the \emph{Bonatti-Viana construction}. The diffeomorphisms constucted this way, and their $C^1$ perturbations are the object of our study. We recall the main points of the construction, referring to \cite{BV} and~\cite{BF} for full details. In \cite{BF}, Buzzi and Fisher added some refinements to the details of the construction, allowing useful additional control which we shall assume in this paper. Let $A \in \mathrm{SL}(4, \mathbb{Z})$ with four distinct real eigenvalues 
\[
      0<\lambda_1< \lambda_2 < 1/3<3<\lambda_3< \lambda_4.
\]
The Bonatti--Viana construction yields diffeomorphisms, which we denote by $\fbv$, which are $C^0$ small, but $C^1$ large, deformations of $f_A$.   

{Recall that $3\eta$ is an expansivity constant for $\fB$. At some points in our analysis (see \S \ref{lps} and \S \ref{sec:bv-expansivity}), we  require that $\eta$ is not too large so that calculations at scales involving $\eta$ are local.
We fix $0< \rho < 3\eta$ and carry out a perturbation in $\rho$-neighbourhoods of $q$ and $q'$. Around $q$ we will deform in the weak stable direction and around $q'$ in the weak unstable direction.  The third fixed point will be left unperturbed to ensure robust transitivity.

Let $F^s,F^u$ be the two-dimensional subspaces of $\RR^d$ corresponding to contracting and expanding eigenvalues of $A$, respectively.  Let $\kappa = 2\bk(F^s,F^u)$, where $\bk$ is as in \eqref{eqn:barkappa}. 

Fixing $\rho>0$, we consider the scales $\rho' = 5\rho$ and $\rho'' = \sixtythree\kappa\rho'$.  We assume that $\rho$ is sufficiently small that $\rho'' < 6\eta$. The role of these scales is as follows:
\begin{enumerate}
\item
 The perturbation takes place in the balls $B(q,\rho)$ and $B(q',\rho)$  -- outside of these balls the new map is identically equal to $\fB$;
\item The scale $\rho'$ is chosen so at this scale the center-stable (resp. center-unstable) leaves are contracted by $g$ (resp. $g^{-1}$);
\item The scale $\rho''$ is the distance that points need to be away from $q$ and $q'$ to guarantee uniform contraction/expansion estimates at a large enough scale to verify the hypotheses of Theorem \ref{t.generalBV}.
\end{enumerate}

\begin{figure}[htb]
\begin{center}
\psfrag{A}{$f_A$}
\psfrag{B}{$ $}
\psfrag{q}{$q$}
\psfrag{r}{$q_1$}
\psfrag{s}{$q$}
\psfrag{t}{$q_2$}
\psfrag{u}{$q_1$}
\psfrag{v}{$q$}
\psfrag{w}{$q_2$}
\psfrag{C}{$\hat{f}$}
\includegraphics[width=.9\textwidth]{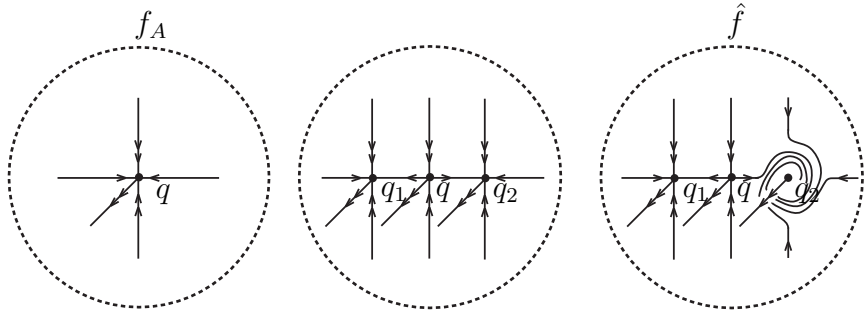}
\caption{Bonatti--Viana construction}\label{f.bv}
\end{center}
\end{figure}

The deformation around the points $q$ and $q'$ is done in two steps, illustrated in Figure \ref{f.bv}.  We describe the deformation around $q$.  
First, we perform a deformation around $q$ in the stable direction $\lambda_2$ as follows. 
Inside $B(q,\rho)$, the fixed point $q$ undergoes a pitchfork bifurcation in the direction corresponding to $\lambda_2$.   

The stable index of $q$ changes from 2 to 1 and two new fixed points $q_1$ and $q_2$ are created. The second step is to deform the diffeomorphism  in a neighborhood of $q_2$ so that the contracting eigenvalues become complex; see Figure~\ref{f.bv}.

Note the creation of fixed points with different indices prevents
the topologically transitive map from being Anosov.
These non-real eigenvalues also forbid the existence of a one-dimensional invariant sub-bundle inside $E^{cs}$.  So the resulting map $\hat f$ has a splitting $E^{cs}\oplus E^{cu}$.

To finish the construction take the deformation just made on $\fB$ near $q$ and repeat it so that the map is equal to $\hat f^{-1}$ in the neighborhood of  $q'$. We obtain a map $\fbv$ that is robustly transitive,  not partially hyperbolic, and has a dominated splitting $T\mathbb T^4=E^{cs}\oplus E^{cu}$ with 
$\dim E^{cs}=\dim E^{cu}=2$ (see~\cite{BV} for proofs of these facts).  

We fix a small $\beta$ and we can ensure in the construction that  $E^{cs} \subset C_{\beta/2}(F^{s},F^{u})$ and $E^{cu}\subset C_{\beta/2}(F^{u},F^{s})$. To simplify computations, we assume explicit upper bounds on $\beta$ at a couple of points in the proof (see e.g. proof of Lemmas \ref{lem:large-scale}, \ref{bowen-ballsBV}, and \ref{lem:Wcu-span}). We may assume that $\beta <1/3$.

Let $C=C(\fB)$ be the constant provided by Lemma \ref{shadowinglemma}. Outside $B(q,\rho) \cup B(q', \rho)$, the maps $\fbv$ and $\fB$ are identical, and we can carry out the construction so there exists a constant $K$ so that both $\fB(B(q,\rho)) \subset B(q, K \rho)$ and $\fbv(B(q,\rho)) \subset B(q, K \rho)$, and similarly for $q'$. Thus the $C^0$ distance between $\fbv$ and $\fB$ is at most $K\rho$. In particular, by choosing $\rho$ small, we can ensure that  $d_{C^0}(\fbv, \fB) < \eta/C$. Thus, we can apply Proposition \ref{pressuredrop} to $\fbv$, or to a perturbation of $\fbv$.

We now consider diffeomorphisms $g$ in a $C^1$ neighborhood of $\fbv$.  We recall results from \cite{BF} on integrability of foliations. We assume that the construction of $\fbv$ is  carried out so that the resulting deformation \emph{respects the domination of }$\fB$. This property is defined in \cite[Definition 2.3]{BF}, and verified for $\fbv$ in \cite[\S 7]{BF}. This is a $C^1$ robust condition which, by Theorem 3.1 of \cite{BF}, ensures integrability of the dominated splitting. Thus,  for $g\in\mathrm{Diff}(\mathbb{T}^4)$ sufficiently close to $\fbv$, there are invariant foliations tangent to $E^{cs}_g$ and $E^{cu}_g$ respectively. Furthermore, the argument of Lemma 6.1 and 6.2 of \cite{BV} shows that each leaf of each foliation is dense in the torus. The existence of foliations was not known when \cite{BV} was written, but these arguments apply with only minor modification now that the existence result has been established by \cite{BF}.   
Thus, we can consider a $C^1$-neighborhood $\VVV$ of $\fbv$ such that the following is true for every $g\in \VVV(\fbv)$:
\begin{itemize}
\item $d_{C^0}(g, \fB) < \eta/C$;
\item $g$ has a dominated splitting $T\TT^4 = E^{cs}_g \oplus E^{cu}_g$, with $\dim E^{cs}_g = \dim E^{cu}_g = 2$ and $E_g^{cs}, E_g^{cu}$ contained in $C_\beta(F^s,F^u)$ and $C_\beta(F^u,F^s)$ respectively; 
\item The distributions $E^{cs}_g$, $E^{cu}_g$ integrate to foliations $W^{cs}_g$, $W^{cu}_g$.
\item Each of the leaves $W^{cs}_g(x)$ and $W^{cu}_g(x)$ is dense for every $x\in \TT^4$.
\end{itemize}
Given $g\in \VVV$, we define the quantities
\[
\begin{aligned}
\lambda_{s}(g) &= \sup \{ \|Dg|_{E^{cs}(x)} \| : x\in \TT^4 \setminus B(q,\rho) \}, \\
\lambda_{u}(g) &= \inf \{ \|Dg|_{E^{cu}(x)}^{-1} \|^{-1} : x\in \TT^4 \setminus B(q',\rho) \}, \\
\lambda_{cs}(g) &= \sup \{ \|Dg|_{E^{cs}(x)} \| : x\in \TT^4\}, \\
\lambda_{cu}(g) &= \inf \{ \|Dg|_{E^{cu}(x)}^{-1} \|^{-1} : x\in \TT^4\}, \\
\lambda(g)&=\max\{\lambda_{cs}(g), \lambda_{cu}(g)^{-1}\}.
\end{aligned}
\]
Note that by the construction of $\fbv$ we have
\begin{align*}
\lambda_s(\fbv) &< 1 < \lambda_{cs}(\fbv), \\
\lambda_{cu}(\fbv) &<1 < \lambda_u(\fbv),
\end{align*}
and we can carry out the construction so that $\lambda(\fbv)$ is arbitrarily close to 1. By continuity, these inequalities hold for $C^1$-perturbations of $\fbv$. { We also have  $\lambda_{s}(g)$ and $\lambda_{u}(g)$ arbitrarily close to the corresponding values for $\fB$.} 
 We let
\begin{equation}\label{def:gamma}
\gamma(g)=\max \left\{ \frac{\log \lambda_{cs}(g)}{\log \lambda_{cs}(g) - \log \lambda_s(g)}, \frac{\log \lambda_{cu }(g)}{\log \lambda_{cu}(g) - \log \lambda_u(g)}. \right\}
\end{equation}
Note that $\gamma(g) \to 0$ as $\lambda(g) \to 1$ (as long as $\lambda_s(g), \lambda_u(g) \not \to 1$).
A simple calculation shows that for any $r>\gamma$, we have
\begin{align}\label{eqn:gamma1}
\lambda_{cs}^{1-r}\lambda_s^r &< 1, \\
\label{eqn:gamma2}
\lambda_{cu}^{1-r}\lambda_u^r &> 1, 
\end{align}
so that in particular, writing
\begin{equation}\label{eqn:theta-r-bv}
\theta_r(g) = \max(\lambda_{cs}^{1-r}\lambda_s^r, \lambda_{cu}^{-(1-r)}\lambda_u^{-r}),
\end{equation}
we have $\theta_r(g)<1$ for all $r>\gamma(g)$.  For notational convenience, we write
\begin{equation}\label{eqn:Q}
Q = B(q,\rho''+\rho) \cup B(q',\rho''+\rho).
\end{equation}
We now state the precise version of Theorem \ref{main2}.

\begin{thm}\label{t.BV}
Given $g\in \VVV(\fbv)$ as above, let $\gamma = \gamma(g)$, $\lambda=\lambda(g)$.  Let 
$\varphi\colon \mathbb{T}^4\to \mathbb{R}$ be H\"older continuous, and set
 $V = \Var(\ph, \sixtythree\rho')$.
Let
\begin{equation}\label{def:phi}
\Phi(\varphi; g)= 6\log \lambda + (1-\gamma) \sup_{Q}\ph + \gamma( \sup_{\TT^4}\ph + \log L + h) +H(2\gamma)  + V.
\end{equation}
If $\Phi(\varphi; g) < P(\ph; g)$, then $\varphi$ has a unique equilibrium state. 
\end{thm}
The $C^1$-open set $\bigcup_{\fbv \in \UUU_{\lambda, \rho}} \VVV(\fbv)$ gives a large class of Bonatti-Viana diffeomorphisms for which this theorem can be used to investigate uniqueness of equilibrium states.

We remark that in the uniformly hyperbolic setting, every H\"older potential is cohomologous to a potential with $\sup \ph < P(\ph)$, which is equivalent to the condition that every equilibrium state for this potential have positive entropy; see \cite[Theorem 6.1]{CFT_Mane}.  Conversely, this condition can sometimes be used beyond uniform hyperbolicity to guarantee that equilibrium states ignore the `bad' part of the system and are unique \cite{DKU}.  Our condition in Theorem \ref{t.BV} is in this spirit; one should not expect to obtain uniqueness for every H\"older potential, so for a result like this, some restriction on the class of potentials is necessary.


\section{Corollaries of Theorem \ref{t.BV}} \label{s:corproof}
Before we prove Theorem \ref{t.BV}, we show how to use it to obtain the two corollaries mentioned in the introduction. 

\begin{cor}\label{t.BV-range}  
Let $\VVV(\fbv)\subset \Diff(\TT^4)$ be as above, and suppose $g\in \VVV(\fbv)$ is such that for $L=L(\fB, \eta)$, $h=\htop(\fB)$, $\gamma=\gamma(g)$, and $\lambda=\lambda(g)$ we have
\begin{equation} \label{hestimate}
6\log \lambda + \gamma(\log L + h) +H(2\gamma) < h.
\end{equation}
Let $V(\varphi) = \Var(\ph, \sixtythree\rho') + \Var(\ph, \eta'),$ where $\eta' = C(\fB)d_{C_0}(\fB, g)$. Then writing 
$$D = h - 6\log\lambda - \gamma(\log L + h) - H(2\gamma) > 0,$$
every H\"older continuous potential $\varphi$ with the 
bounded range hypothesis 
$\sup \varphi - \inf \varphi +V(\varphi) <D$
has a unique equilibrium state. In particular, \eqref{hestimate} is a sufficient criterion for $g \in \VVV(\fbv)$ to have a unique MME. 
\end{cor} 

\begin{proof}
If $\sup \ph - \inf \ph +V(\ph)< D$ 
, then
\begin{align*}
6 \log \lambda + (1-\gamma) &\sup_{Q}\ph + \gamma( \sup_{\TT^4}\ph + h + \log L) +H(2\gamma) +V \\
&= (1-\gamma) \sup_{Q}\ph + \gamma(\sup_{\TT^4}\ph) + \htop(\fB) +V - D \\
&\leq \sup_{\TT^4} \ph + \htop(\fB) +V - D\\
&< \inf\ph + \htop(\fB) - \Var(\ph, \eta') \\ 
&\leq P(\ph; \fB)- \Var(\ph, \eta')\leq P(\varphi; g).
\end{align*}
Thus Theorem \ref{t.BV} applies.
\end{proof} 
Since $V(\varphi)\leq 2(\sup \varphi - \inf \varphi)$, we could remove the variance term in our bounded range hypothesis by asking that $3(\sup \varphi - \inf \varphi)  <D$. 


\begin{cor}\label{cor1.2}
Let $\varphi\colon \mathbb{T}^4\to \mathbb{R}$ be a H\"older continuous potential. In any $C^0$-neighborhood of $\fB$, there exists a $C^1$-open subset $\mathcal{V}_0 \subset \Diff(\mathbb{T}^4)$ containing diffeomorphisms from the Bonatti--Viana family such that for every $g\in \mathcal{V}_0$, $g$ has a dominated splitting and is not partially hyperbolic and $(\mathbb{T}^4,g,\varphi)$ has a unique equilibrium state.
\end{cor}

\begin{proof}
A diffeomorphism $\fbv\in \UUU_{\lambda, \rho}$ can be found in any $C^0$ neighbourhood of $\fB$ by taking $\rho$ to be small. Let $\VVV= \VVV(\fbv)$, and $\VVV_0$ be the set of $g\in \VVV$ such that $\Phi(g; \varphi) < P(\ph; g)$.  Note that $\Phi(g; \varphi)$ is continuous under $C^1$ perturbation of $g$, so $\VVV_0$ is $C^1$-open.  It only remains to show that $\VVV_0$ is non-empty when $\rho$ and $\log \lambda$ are sufficiently small.

Let $\eta' = C(\fB) d_{C_0}(g, \fB)$. Recall from Proposition \ref{pressuredrop}\ref{Pg-geq} that $P(\varphi; g)\geq P(\ph; \fB) - \Var(\ph, \eta')$.  Moreover, we have
\[
(1-\gamma) \sup_{Q}\ph  \leq \max\{\ph(q), \ph(q')\}+ \Var(\ph, 2 \rho'').
\]
Thus to prove $\Phi(\ph; g) < P(\ph; g)$  it suffices to verify that 
\[
\max\{\ph(q), \ph(q')\} + 
6\log \lambda  + \gamma( \sup_{\TT^4}\ph + h + \log L) + H(\gamma)+ V' < P(\ph; f_A),
\]
where $V' = V + \Var(\ph, 2 \rho'')+ \Var(\ph, \eta')$. The scales which appear in the $V'$ term all tend to $0$ as $\rho$ tends to $0$. Given a hyperbolic toral automorphism $\fB$ and a H\"older potential $\ph\colon \TT^4\to \RR$, it is well known that $\ph$ has a unique equilibrium state with the Gibbs property. For a fixed point $p$, the Dirac measure $\delta_p$ clearly does not have the Gibbs property, so cannot be an equilibrium state for $\ph$, and thus
\[
\ph(p) = h_{\delta_p}(\fB) + \int \varphi\, d\delta_p < P(\ph; \fB).
\]
Thus, $\max\{\ph(q), \ph(q')\}< P(\ph; \fB)$. By choosing $\log \lambda$ and $\rho$ small, we can ensure that $\gamma$ and $V'$ are small enough so that the required inequality holds. Thus, $\mathcal V_0$ is non-empty. 
\end{proof}

\section{Proof of the Main Result} \label{s:mainproof}
We now build up a proof of our main result Theorem \ref{t.BV}, which is the more precise version of Theorem \ref{main2}.

\subsection{Local product structure}\label{lps}
We now establish local product structure at scale $6 \eta$ for maps $g \in \VVV$. The assumptions that allow us to do this are that $E^\sigma_g \subset C_\beta^\sigma$ for $\sigma\in \{cu,cs\}$ and that $\beta, \eta$ are not too large.

\begin{lem}\label{lem:lps-mane}
Every $g\in \VVV$ has a local product structure for $W^{cs}_g,W^{cu}_g$ at scale $6\eta$ with constant $\kappa = 2\bk(F^{s},F^u)$.
\end{lem}
\begin{proof}
Let $\widetilde W^{cs}$ and $\widetilde W^{cu}$ be the lifts of $W^{cs}_g,W^{cu}_g$ to $\RR^4$.  Given $x,y\in \TT^4$ with $\eps := d(x,y) < 6\eta$, let $\tilde x,\tilde y\in \RR^4$ be lifts of $x,y$ with $\eps=d(\tilde x,\tilde y)<6\eta$.  By Lemma \ref{lem:Wlps} the intersection $\widetilde W^{cs}(x) \cap \widetilde W^{cu}(y)$ has a unique point $\tilde z$, which projects to $z\in \TT^4$.  Moreover, the leaf distances between $\tilde x,\tilde z$ and $\tilde y, \tilde z$ are at most $(\frac{1+\beta}{1-\beta}) \bk(F^{s},F^u)\eps$.  Since $\beta<\frac 13$ this is less than $2\bk(F^{s},F^u) \epsilon$, so $z$ is in the intersection of the local leaves $(W^{cs}_g)_{\kappa \eps}(x)$ and $ (W^u_g)_{\kappa \eps}(x)$.

{ By choosing $\eta$ not too large, we can ensure that $6 \eta \kappa$ is not too large relative to the diameter of $\TT^4$, so that the projection of $\widetilde W^{cs}_{6 \eta \kappa}(x) \cap \widetilde W^{cu}_{6 \eta \kappa}(y)$ coincides with $W^{cs}_{6 \eta \kappa}(x) \cap W^{cu}_{6 \eta \kappa}(y)$. Thus,  $z$ is the only point in this intersection.}
\end{proof}
\subsection{Specification}\label{sec:spec}

We must control the size of local leaves of $W^{cs}, W^{cu}$ under iteration, and the time to transition from one orbit to another. We use the following fact, which we prove in \S\ref{s.lemmas}.

\begin{lem}\label{lem:WcuWcs}
For every $\delta>0$ there is $R>0$ such that for all $x,y\in \TT^4$, we have
$
W^{cu}_R (x) \cap W^{cs}_\delta (y) \neq \emptyset.
$
\end{lem}

Although the leaves $W^{cu}(x)$ are not expanding at every point, and the leaves $W^{cs}(x)$ are not contracting at every point, we nevertheless see  expansion and contraction if we look at a scale suitably large relative to $\rho$. 
More precisely, consider the quantities $\theta_{cs} = \frac 45 + \frac 15\lambda_s(g) < 1$ and $\theta_{cu} = \frac 45 + \frac 15 \lambda_u(g)^{-1} < 1$.  Let $d_{cs}$ and $d_{cu}$ be the metrics on the leaves $W^{cs}$ and $W^{cu}$.  Then we have the following result.


\begin{lem}\label{lem:large-scale}
If $x\in \TT^4$ and $y\in W^{cs}(x)$ are such that $d_{cs}(x,y) > \rho'$, then $d_{cs}(gx,gy) < \theta_{cs} d_{cs}(x,y)$.  Similarly, if $y\in W^{cu}(x)$ and $d_{cu}(x,y) > \rho'$, then $d_{cu}(g^{-1}x,g^{-1}y) < \theta_{cu} d_{cu}(x,y)$.
\end{lem}
\begin{proof}
We give the proof for $W^{cs}$; the proof for $W^{cu}$ is analogous.  Given a path $\sigma$ on $\TT^4$, write $\ell(\sigma)$ for the length of $\sigma$.  Let $\sigma$ be a path from $x$ to $y$ in $W^{cs}(x)$ such that $\ell(\sigma) = d_{cs}(x,y)$.  Decompose $\sigma$ as the disjoint union of paths $\sigma_i$ where $\ell(\sigma_i) \in [\rho', 2\rho']$.  Clearly it suffices to show that $\ell(g\sigma_i) < \theta_{cs} \ell(\sigma_i)$ for each $i$. {We may assume that $\beta$ is chosen not too large so that}
\begin{equation}\label{eqn:beta-lambda}
(1+\beta)\left(\frac{\lambda(g) -\lambda_s(g)}{1-\lambda_s(g)}\right) < 2
\end{equation}
{We may assume that the path $\sigma_i$ has at most one connected component that intersects $B(q,\rho)$, since $\rho$ and $\ell(\sigma_i) \leq 2\rho'$ are not large enough to wrap around the torus.}
Let $\ell_1$ be the length of this component; because this component lies in $W^{cs}(x)$, which is contained in $C_\beta(F^{s},F^u)$, we have $\ell_1 \leq 2\rho(1+\beta)$.  Let $\ell_2 = \ell(\sigma_i) - \ell_1$.  Let $v$ be a tangent vector to the curve $\sigma$ at the point $p\in \TT^4$.  If $p\in B(q,\rho)$ then we have $\|Dg(v)\| \leq \lambda(g) \|v\|$, while if $p\notin B(q,\rho)$ then $\|Dg(v)\| \leq \lambda_s(g) \|v\|$.  Thus we obtain
\begin{align*}
\ell(g\sigma_i) &\leq \lambda \ell_1 + \lambda_s \ell_2
= (\lambda - \lambda_s)\ell_1 + \lambda_s \ell(\sigma_i) \\
&\leq (\lambda - \lambda_s)2\rho(1+\beta) + \lambda_s \ell(\sigma_i)
< 4(1-\lambda_s) \rho + \lambda_s \ell(\sigma_i),
\end{align*}
where the last inequality uses \eqref{eqn:beta-lambda}. Since $\rho = \frac 15 \rho' \leq \frac 15 \ell(\sigma_i)$, this gives
\[
\ell(g\sigma_i) < \tfrac 45 (1-\lambda_s) \ell(\sigma_i) + \lambda_s \ell(\sigma_i) = \theta_{cs} \ell(\sigma_i).
\]
Summing over $i$ gives $d_{cs}(gx,gy) \leq \ell(g\sigma) < \theta_{cs} \ell(\sigma) = \theta_{cs} d_{cs}(x,y)$. The proof for $d_{cu}$ is similar.
\end{proof}

The following is an immediate consequence of Lemmas \ref{lem:large-scale} and \ref{lem:WcuWcs}.

\begin{lem}\label{lem:overflowing}
For every $R > \rho'$ and $x\in \TT^4$, we have
\begin{align*}
g(W_{R}^{cs}(x)) &\subset W_{\theta_{cs} R}^{cs}(gx), \\
g^{-1}(W_{R}^{cu}(x)) &\subset W_{\theta_{cu}^{-1} R}^{cu}(g^{-1}x).
\end{align*}
In particular, there is $\tau_0\in \NN$ such that for every $x,y\in \TT^4$ we have
\begin{equation}\label{eqn:uniform-mix}
g^{\tau_0}(W_{\rho'}^{cu}(x)) \cap W_{\rho'}^{cs}(y) \neq \emptyset.
\end{equation}
\end{lem}

Let $\rho'':=  \sixtythree \kappa \rho' $, where $\kappa = 2\bk(F^{s},F^{u})$ is the constant arising in the local product structure of $W^{cs}, W^{cu}$.
Let $\chi$ be the indicator function of $\TT^4 \setminus B(q, \rho''+\rho)$ and $\chi'$ be the indicator function of $\TT^4 \setminus B(q', \rho''+\rho)$.  
The scale $\rho''+\rho$ is chosen to ensure uniform  estimates on $W^{cs}_{\rho''}$ and $W^{cu}_{\rho''}$ for points with $\chi(x)=1$ and $\chi'(x)=1$.

From now on we fix $r>\gamma(g)$, and consider the following collection of orbit segments:
\[
\GGG = \{(x,n): \tfrac{1}{i}S_i\chi(x) \geq r  \mbox { and } \tfrac{1}{i} S_i \chi' (f^{n-i} x) \geq r \text{ for all } 0\leq i\leq n \}.
\]
We will show that $\GGG^M$ has specification at scale $3\rho'$.
To get a decomposition we consider $\GGG$ together with the collections
\begin{align*}
\PPP &= \{(x,n)\in \mathbb{T}^4\times \mathbb{N}\, :\, \tfrac{1}{n}S_n\chi(x)< r \},\\
 \SSS &= \{(x,n)\in \mathbb{T}^4\times \mathbb{N}\, :\, \tfrac{1}{n}S_n\chi'(x)< r \}.
\end{align*}

\begin{figure}[htb]
\includegraphics[width=.9\textwidth]{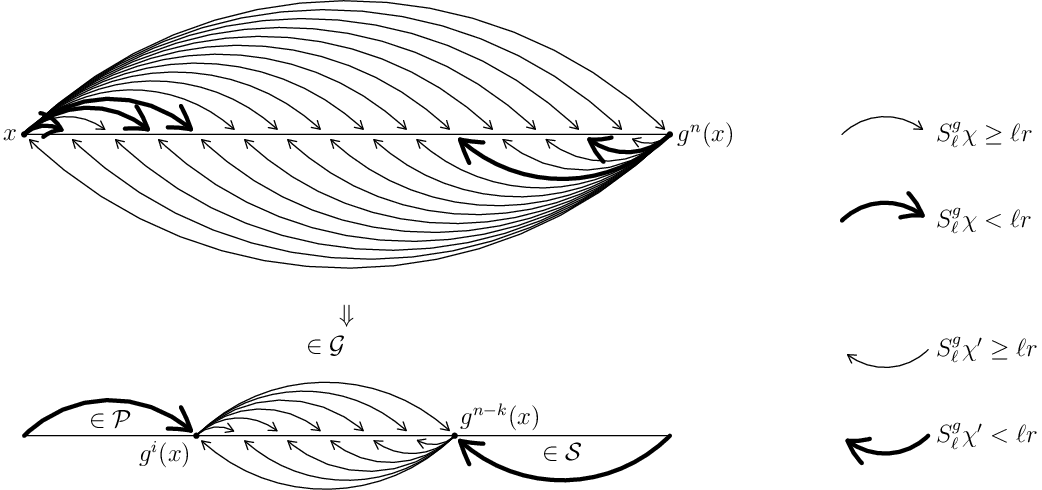}
\caption{Decomposing an orbit segment}\label{f.BVdecomp}
\end{figure}

\begin{lem} \label{BVdecomp}
The collections $\PPP,\GGG,\SSS$ form a decomposition for $g$.
\end{lem}
\begin{proof}
Let $(x, n) \in X \times \NN$. Let $0 \leq i \leq n$ be the largest integer so $\frac{1}{i}S_i\chi(x)< r$, and $0 \leq k \leq n$ be the largest integer so $\frac{1}{k}S_k\chi'(g^{n-k}x)< r$. A short calculation shows that $\frac{1}{\ell}S_\ell\chi(g^ix)\geq r$ for $0\leq \ell \leq n-i$, and $\frac{1}{\ell}S_\ell\chi'(g^{n-k-\ell}x)\geq r$ for $0\leq \ell\leq n-k$, see Figure~\ref{f.BVdecomp}. Thus, if we assume that $i+k<n$, letting $j= n-(i+k)$, we have
\[
(x, i) \in \PPP,\qquad (g^ix, j) \in \GGG,\qquad (g^{i+j}x, k) \in \SSS.
\]
If $i+k\geq n$, we can choose a decomposition with $j=0$. 
\end{proof}
Orbit segments in $\GGG^M$, which is the set of orbit segments $(x,n)$ for which $p \leq M$ and $s\leq M$, satisfy the following.
\begin{lem} \label{unifexpansion}
Let $\nu = \lambda / \theta_r$.  For every $M\in \NN_0$, $(x,n)\in \GGG^M$, and $0\leq i\leq n$, we have
\begin{enumerate}[label=\upshape{(\alph{*})}]
\item $\|Dg^i|_{E^{cs}(y)}\| \leq \nu^{2M}\theta_r^i$ for $y\in B_n(x,\rho'')$;
\item $\|Dg^{-i}|_{E^{cu}(g^ny)}\| \leq \nu^{2M}\theta_r^i$ for $y\in B_n(x,\rho'')$;
\item $d_{cs}(g^iy,g^iz) \leq \nu^{2M}\theta_r^i d_{cs}(y,z)$ when $y\in B_n(x, \rho')$ and $z\in W_{2\rho'}^{cs}(y)$;
\item $d_{cu}(g^{n-i}y,g^{n-i}z) \leq \nu^{2M}\theta_r^i d_{cu}(y,z)$ when $y\in B_n(g^{-n}x, \rho')$ and $z\in  W_{2\rho'}^{cu}(y)$.
\end{enumerate}
\end{lem}
\begin{proof}
We prove (a). Given $(x,n) \in \GGG^M$ and $0\leq i \leq n$, we have $S_i \chi(x) > ir-2M$, and so the orbit segment $(x, i)$ spends greater than 
$ir-2M$ iterates outsides $B(q, 4 \rho')$, and thus $(y, i)$ spends greater than 
$ir-2M$  iterates outsides $B(q, \rho)$. It follows that
\begin{align*}
\|Dg^i|_{E^{cs}(y)}\| & \leq \lambda^{i- (ir-2M)}\lambda_s^{ir-2M} \\
& = \lambda^{i(1-r)} \lambda_s^{ir} \lambda^{2M}\lambda_s^{-2M}  
 = (\theta_r)^i \nu^{2M}.
\end{align*}
For (c), note that if $y\in B_n(x, \rho')$ and $z' \in W_{2\rho'}^{cs}(y)$, then $z' \in B_n(x, 3 \rho')$. Thus, the uniform derivative estimate of (a) applies to all points in $W_{2\rho'}^{cs}(y)$, and it is an easy exercise to use this to obtain the statement of (c). The proof of (b) is similar to (a), and (d) follows.
\end{proof}

We use the following facts for our result on the specification property:
\begin{itemize}
\item For any $x \in \TT^4$ and $n \in \NN$, from Lemma \ref{lem:overflowing} we have $W_{\rho'}^{cs}(x) \subset B_n(x, \rho')$ and $g^{-n}(W_{\rho'}^{cu}(g^nx)) \subset B_n(x, \rho')$;
\item If $(x, n) \in \GGG^M$ and $y,z \in B_n(x, 3 \rho')$ and $g^nz \in W^{cu}(g^ny)$, then Lemma \ref{unifexpansion} (c) gives $d_n(y, z) \leq \nu^{2M} d_{cu}(g^ny, g^nz)$ and $d_{cu}(y,z) \leq \nu^{2M} \theta_r^n d_{cu}(g^ny,g^nz)$.

\end{itemize}

Given $M$, we take $N=N(M)$ such that $\theta_r^N \nu^{2M} \lambda^{\tau_0} < \frac 12$, where $\tau_0$ is as in \eqref{eqn:uniform-mix}.  Then we let $\GGG^M_{\geq N} := \{(x,n)\in \GGG^M \mid n\geq N\}$.

\begin{lem}\label{lem:GMspec}
For every $M$, let $N=N(M)$ be as above. Then $\GGG^M_{\geq N}$ has specification at scale $3\rho'$.
\end{lem}
\begin{proof}
Write $\tau=\tau_0$, so that \eqref{eqn:uniform-mix} gives $g^\tau(W_{\rho'}^{cu}(x)) \cap W_{\rho'}^{cs}(y) \neq\emptyset$ for every $x,y\in \TT^4$.

For every $(x,n)\in \GGG^M_{\geq N}$ and $y,z\in g^{-(n+\tau)}(g^\tau(W_{\rho'}^{cu}(x)))$, our choice of $N$ gives
\begin{equation}\label{eqn:yz-contracting}
d(y,z) < \tfrac{1}{2}d(g^{n+\tau}y, g^{n+\tau}z).
\end{equation}
Now we show that $\GGG^M_{\geq N}$ has specification with transition time $\tau$.  Given any
$(x_1, n_1), \dots, (x_k, n_k)\in \mathcal{G}^M$ with $n_i \geq N$, we construct $y_j$ iteratively such that $(y_j,m_j)$ shadows $(x_1,n_1),\dots,(x_j,n_j)$, where $m_1 =n_1$, $m_2 = n_1 + \tau + n_2$, $\dots$, $m_k = (\sum_{i=1}^{k} n_i) + (k-1)\tau$. We also set $m_{0} = - \tau$.
 
Start by letting $y_1 = x_1$, and we choose $y_2,\dots, y_k$ iteratively so that
$$
\begin{matrix}
g^{m_1}y_2 \in W^{cu}_{\rho'} (g^{m_1}y_1) &  \mbox{and} & g^{m_1+\tau}y_2 \in W^{cs}_{\rho'} (x_2)\\
g^{m_2}y_3 \in W^{cu}_{\rho'} (g^{m_2}y_2) &  \mbox{and}   & g^{m_2+\tau}y_3 \in W^{cs}_{\rho'} (x_3)\\
\vdots & \vdots & \vdots \\
g^{m_{k-1}}y_k \in W^{cu}_{\rho'} (g^{m_{k-1}}y_{k-1}) &  \mbox{and}   & g^{m_{k-1}+\tau}y_{k} \in W^{cs}_{\rho'} (x_k).\\
\end{matrix}
$$
That is, for $j \in \{1, \ldots, k-1\}$, we let $y_{j+1}$ be a point such that 
\[
y_{j+1} \in g^{-m_j}(W^{cu}_{\rho'}(g^{m_j}y_j))\cap g^{-(m_j+ \tau)}( W^{cs}_{\rho'}(x_{j+1})).
\]
Using the fact that $g^{m_j}y_{j+1}$ is in the centre-unstable manifold of $g^{m_j}y_j$ together with the estimate \eqref{eqn:yz-contracting}, we obtain that 
$$
\begin{matrix}
d_{n_j}(g^{m_{j-1}+\tau}y_j, g^{m_{j-1}+\tau}y_{j+1})& < &\rho' \\
d_{n_{j-1}}(g^{m_{j-2}+\tau}y_j, g^{m_{j-2}+\tau}y_{j+1})& < &\rho'/2 \\
\vdots &  &\vdots\\
d_{n_1}(y_j, y_{j+1}) & < & \rho'/2^{j-1}.
\end{matrix}
$$
That is, $d_{n_{j-i}}(g^{m_{j-i-1}+\tau}y_j, g^{m_{j-i-1}+\tau}y_{j+1}) < \rho'/2^i$ for $i \in \{0, \ldots, j-1\}$. This estimate, together with the fact that $g^{m_j+\tau}(y_{j+1}) \in B_{n_{j+1}}(x_{j+1}, \rho')$ from Lemma \ref{lem:overflowing} gives that $d_{n_j}(g^{m_{j-1}+ \tau}y_k, x_j) < 2 \rho' + \sum_{j=1}^\infty 2^{-j} \rho' = 3\rho'.$  It follows that
$$y_k\in \bigcap_{j=1}^k g^{-(m_{j-1} + \tau)}B_{n_j}(x_j, 3 \rho'),$$
and thus $\mathcal{G}^M_{\geq N}$ has specification at scale $3 \rho'$.
\end{proof}


\subsection{Bowen property}

Let $\theta_r \in (0,1)$ be the constant that was defined at \eqref{eqn:theta-r-bv}, and let $\kappa$ be the constant associated with the local product structure of $E^{cs}_g \oplus E^{cu}_g$.

\begin{lem}\label{bowen-ballsBV}
Given $(x,n)\in \GGG$ and $y\in B_n(x,\sixtythree\rho')$, we have
\begin{equation}\label{eqn:hyp-hyp2}
d(g^kx,g^ky) \leq \kappa \onetwosix\rho' (\theta_r^{n-k} + \theta_r^{k})
\end{equation}
for every $0\leq k\leq n$.
\end{lem}
\begin{proof}
Using the local product structure at scale $\sixtythree\rho'$ 
and observing that $\rho'' = \kappa\sixtythree\rho'$, we see that for each $0\leq k\leq n$ there is $z_k\in W_{\rho''}^{cs}(g^k x) \cap W_{\rho''}^{cu}(g^k y)$.  By invariance of the foliations we get $z_k = g^k(z_0)$.
Let $\gamma$ be the geodesic in $W^{cu}(g^n y)$ that connects $g^n(y)$ and $z_n$.  Since each endpoint of $\gamma$ is in $B(g^n(x),\rho'')$, convexity implies that the straight line joining them lies in $B(g^n(x),\rho'')$; choose $\beta$ small enough that the leaves $W^{cu}$ are close enough to linear that the same is true for $\gamma$, no matter what $x,n$ we choose. Then we can apply Lemma \ref{unifexpansion}(b) along $\gamma$ to obtain 
\[
d(z_k,g^ky) \leq \theta_r^{n-k} d(z_n,g^ny) \leq \theta_r^{n-k} \kappa \sixtythree\rho',
\]
and Lemma \ref{unifexpansion}(a) gives $d(g^kx,z_k) \leq \theta_r^k d(x,z_0) \leq \theta_r^k \kappa \sixtythree\rho'$.
The result follows.
\end{proof}

\begin{lem} \label{bowenpropBV}
Any H\"older continuous $\varphi$ has the Bowen property on $\GGG$ at scale $\sixtythree\rho'$.
\end{lem}
\begin{proof}
By H\"older continuity, there are constants $K>0$ and $\alpha\in (0,1)$ such that $|\ph(x) - \ph(y)| \leq K d(x,y)^\alpha$ for all $x,y\in \TT^d$.  Now given  $(x, n) \in \GGG$ and $y\in B_n(x, \sixtythree\rho')$, Lemma \ref{bowen-ballsBV} gives
\begin{align*}
|S_n\ph(x)-S_n \ph(y)| &\leq K\sum_{k=0}^{n-1}d(g^kx, g^ky)^\alpha \leq K (\kappa \onetwosix\rho')^\alpha \sum_{k=0}^{n-1} (\theta_r^{n-k} + \theta_r^{k})^\alpha \\
&\leq 2^\alpha K (\kappa \onetwosix\rho')^\alpha\sum_{j=0}^\infty (\theta_r^{j\alpha} + \theta_r^{j\alpha})
=: V <\infty.\qedhere
\end{align*}
\end{proof}

\subsection{Expansivity}\label{sec:bv-expansivity}

We want to obtain a bound on $h_g^*$, the tail entropy of $g$. By results of \cite{DF}, the tail entropy may be positive. 
We assume that $\beta$ is chosen not too large so that $(1+\beta)/(1-\beta)<2$.

\begin{lem}\label{lem:Wcu-span}
Let $\delta \in (0, 6 \eta)$.
Given $n\in \NN$, and $x,z\in \TT^4$ such that $d_n(x,z) < 6\eta$, we have
\begin{equation}\label{eqn:Wcu-span}
\Lspan_n(W^{cu}_{6\eta}(z) \cap B_n(x,6\eta), 0, \delta; g) \leq 32(6\eta)^2 \delta^{-2} \lambda^{2n}.
\end{equation}
\end{lem}
\begin{proof}
Write $\epsilon=6\eta$. First we prove that 
\[
W^{cu}_{\epsilon}(z) \cap B_{k}(x,\epsilon) \subset g^{-(k-1)}(W_{4\epsilon }^{cu}(g^{k-1}z))
\]
for $k \in \{1, \ldots, n\}$.  This follows by induction; it is true for $k=1$, and given the result for $k\in\{1, \ldots, n-1\}$, we see that any $z'\in W_{\epsilon}^{cu}(z) \cap B_{k+1}(x,\epsilon)$ has $g^{k-1}(z')\in W_{4 \epsilon}^{cu}(g^{k-1}z)$ by the inductive hypothesis, and so
\[
g^{k}(z') \in W_{4 \epsilon \|Dg\|}^{cu}(g^{k}z).
\]
Also $g^k(z') \in B(g^kx, \epsilon) \subset B(g^kz, 2 \epsilon)$, where the last inclusion follows because $4 \epsilon \|Dg\|$ is not enough distance to wrap all the way around the torus and enter $B(g^{k}x,\epsilon)$ again. This is true because $\epsilon$ is assumed to be not too large. This is the only requirement on $\epsilon$ in this proof.
Thus, by Lemma \ref{compare:dist},  $g^{k}(z') \in W_{2\epsilon (1+\beta)/(1-\beta)}^{cu}(g^{k}z) \subset W_{4\epsilon }^{cu}(g^{k}z)$. 
Now fix $\alpha=\delta(1+\beta)^{-1}\lambda^{-n}$.  Recall that $W_{4\eps}^{cu}(g^nz)$ is the graph of a function from $F^{cu}$ to $F^{cs}$ with norm less than $\beta$.  The projection of $W_{4\eps}^{cu}(g^nz)$ to $F^{cu}$ along $F^{cs}$ is contained in a ball of radius $4\eps$, so $B_{4\eps}(0)$ in $F^{cu}$  has an $\alpha$-dense subset in the $d_n$-metric with cardinality less than or equal to $ 16\eps^2\alpha^{-2}$.  Projecting this set back to $W_{4\eps}^{cu}(g^nz)$ along $F^{cu}$ gives $E\subset W_{4\eps}^{cu}(g^nz)$ that is $(1+\beta)\alpha$-dense.

Consider the set $g^{-n}(E) \subset W^{cu}(z)$.  Given any $y\in W^{cu}_{\eps}(z) \cap B_n(x,\eps)$, we have $g^n(y) \in W_{4\eps}^{cu}(g^nz)$ and so there is $z'\in E$ such that $d_{cu}(g^ny, g^nz') < (1+\beta)\alpha$.  Since $g^{-1}$ expands distances along $W^{cu}$ by at most $\lambda$, we have $d_n(y,z') < (1+\beta)\alpha \lambda^n$.  We see that $g^{-n}(E)$ is an $(n,\delta)$-spanning set for $W^{cu}_{\eps}(z) \cap B_n(x,\eps)$, and moreover
\[
\#g^{-n}(E) \leq 16\eps^2\alpha^{-2} \leq 16\eps^2 \delta^{-2} (1+\beta)^2 \lambda^{2n},
\]
which gives \eqref{eqn:Wcu-span} and completes the proof of Lemma \ref{lem:Wcu-span}.
\end{proof}


\begin{lem} \label{hexp}
For every $g\in \VVV$ we have $h_g^\ast(6\eta) \leq 6\log \lambda$.
\end{lem}
\begin{proof}
Given $x\in \TT^4$ and $\delta>0$, we estimate $\Lspan_n(\Gamma_{6\eta}(x),0,2\delta;g)$ for $n\in \NN$.  To do this, we start by fixing
\begin{equation}\label{eqn:alpha-small}
\alpha = \alpha(n) = \frac \delta{\kappa \lambda^n}
\end{equation}
where $\kappa$ is from the local product structure.
Let $E\subset \Gamma_{6\eta}(x)$ be an $\alpha$-dense set with cardinality 
\[
\#E \leq (12\eta/\alpha)^4 = (12\eta)^4 \kappa^4 \lambda^{4n} \delta^{-4};
\]
note that such a set exists because $\Gamma_{6\eta}(x)$ is contained in $x + [-6\eta,6\eta]^4$.

Now we have $W_{\kappa\alpha}^{cu}(z) \subset W_{6\eta}^{cu}(z)$ for each $z\in E$, so by Lemma \ref{lem:Wcu-span}, there is an $(n,\delta)$-spanning set $E_z$ for $W_{\kappa\alpha}^{cu}(z) \cap B_n(x,6\eta)$ with 
$$\#E_z \leq 32(6\eta)^2 \delta^{-2} \lambda^{2n}.$$  Let $E' = \bigcup_{z\in E} E_z$, then we have
\[
\# E' \leq 32(12\eta)^6 \delta^{-6} \kappa^4 \lambda^{6n}.
\]
We claim that $E'$ is $(n,2\delta)$-spanning for $\Gamma_{6\eta}(x)$, which will complete the proof of Lemma \ref{hexp}.  To see this, take any $y\in \Gamma_{6\eta}(x)$, and observe that because $E$ is $\alpha$-dense in $B(x,6\eta)$, there is $z = z(y)\in E \cap B(y,\alpha)$.  By the local product structure there is $\bar z = \bar z(y) \in W^{cs}_{\kappa\alpha}(y) \cap W^{cu}_{\kappa\alpha}(z)$.  Notice that because distance expansion along $W^{cu}$ is bounded above by $\lambda$ for each iteration of $g$, we have
\begin{equation}\label{eqn:dnyz'}
d_n(y,\bar z) < \kappa \alpha \lambda^n = \delta.
\end{equation}
By our choice of $E_z$, there is $z'\in E_z$ such that $d_n(z',\bar z) < \delta$.  Thus $d_n(y,z') < 2\delta$, as required.  It follows that
\[
\Lspan_n(\Gamma_{6\eta}(x),0,2\delta;g) \leq 32(12\eta)^6 \delta^{-6} \kappa^4 \lambda^{6n},
\]
hence $h_g^*(6\eta) \leq 6\log \lambda$, which proves Lemma \ref{hexp}.
\end{proof}

\begin{lem} \label{BV-satisfieshyp}
For every $r > \gamma(g)$ and $\eps = \sixtythree\rho'$, the diffeomorphism $g$ satisfies Condition \ref{E}.
\end{lem}
\begin{proof}
Suppose $x\in \TT^4$, $r>0$, and $n_k,m_k\to \infty$ are such that
\begin{equation}\label{eqn:nkmk}
\tfrac 1{n_k} S_{n_k}^g \chi(x) \geq r,
\qquad
\tfrac 1{m_k} S_{m_k}^{g^{-1}} \chi'(x) \geq r
\end{equation}
for every $k$.  Our goal is to show that $\Gamma_{\eps}(x) = \{x\}$.

First we fix $r' \in (\gamma,r)$ and observe that by Pliss' lemma \cite{Pliss} there are $m_k', n_k'\to\infty$ such that
\begin{equation}\label{eqn:nkmk'}
\begin{aligned}
S_m^g\chi'(g^{-m_k'}x) &\geq mr' \text{ for every } 0 \leq m \leq m_k', \\
S_n^{g^{-1}}\chi(g^{n_k'}x) &\geq nr' \text{ for every } 0 \leq n \leq n_k'.
\end{aligned}
\end{equation}
As in the proof of Lemma \ref{unifexpansion},
for every $y\in B_{m_k'}(g^{-m_k'}x,\rho'')$ and $z\in g^{n_k'} B_{n_k'}(x,\rho'')$, we now have
\begin{equation}\label{eqn:nkmk''}
\begin{aligned}
\|Dg^m(y)|_{E^{cs}}\| &\leq \theta_{r'}^m \text{ for every } 0 \leq m \leq m_k', \\
\|Dg^{-n}(z)|_{E^{cu}}\| &\leq \theta_{r'}^m \text{ for every } 0 \leq n \leq n_k',
\end{aligned}
\end{equation}
where $\theta_{r'} < 1$ is as in \eqref{eqn:theta-r-bv}.

Now let $x'\in \Gamma_{\eps}(x)$.  By the local product structure, and $\epsilon$ being not too large, there is a unique point $x'' \in W_{\kappa\eps}^{cu}(x) \cap W_{\kappa\eps}^{cs}(x')$. Applying $g$ we see that
\[
g(x'') \in W_{\kappa \eps\|Dg\|}^{cs}(gx) \cap W_{\kappa \eps\|Dg^{-1}\|}^{cu}(gx').
\]
But by the local product structure, $W_{\kappa \eps \|Dg\|}^{cs}(gx)$ and $W_{\kappa \eps\|Dg^{-1}\|}^{cu}(gx')$ have a unique intersection point if $\max \{\kappa \eps \|Dg\|, \kappa \eps\|Dg^{-1}\|\}< 6 \eta$. 
Thus $g(x'')$ is the unique intersection point, and since $d(gx, gx')\leq \eps$, it follows that
$g(x'') \in W_{\kappa \eps}^{cs}(gx) \cap W_{\kappa \eps}^{cu}(gx')$.
Iterating the above argument gives for every $n\in \ZZ$,
\begin{equation}\label{eqn:gnx''}
g^n(x'') \in W_{\kappa\eps}^{cu}(g^n x) \cap W_{\kappa\eps}^{cs}(g^n x').
\end{equation}
In particular, for each $k\in \NN$ we can apply \eqref{eqn:nkmk''} with $z$ a point along the $W^{cu}$-geodesic from $g^{n_k'}x$ to $g^{n_k'}x''$, and deduce that
\[
d_{cu}(x,x'') \leq \theta_{r'}^{n_k'} d_{cu}(g^{n_k'}x,g^{n_k'}x'') \leq \theta_{r'}^{n_k'} \kappa\eps.
\]
Sending $k\to\infty$ gives $d_{cu}(x,x'') = 0$ and hence $x''=x$ since $x''\in W^{cu}_{\kappa\eps}(x)$.  Now by \eqref{eqn:gnx''} we have $g^n x\in W_{\kappa\eps}^{cs}(g^nx')$ for all $n\in \ZZ$, and for each $k\in \NN$ we can apply \eqref{eqn:nkmk''} with $y$ a point along the $W^{cs}$-geodesic from $g^{-m_k'}x$ to $g^{-m_k'}x'$, obtaining
\[
d_{cs}(x,x') \leq \theta_{r'}^{m_k'} d_{cs}(g^{-m_k'}x, g^{-m_k'}x') \leq \theta_{r'}^{m_k'} \kappa\eps.
\]
Again, as $k$ increases we get $d_{cs}(x,x')=0$ hence $x'=x$, which completes the proof of Lemma \ref{BV-satisfieshyp}.
\end{proof}

\subsection{Verification of Theorem \ref{t.BV}} \label{verification}
We now have all the ingredients to show that if $g\in \VVV(\fbv)$ and $\ph\colon \TT^4\to \RR$ satisfy the hypotheses of Theorem \ref{t.BV}, then the conditions of Theorem \ref{t.generalBV} are satisfied, and hence there is a unique equilibrium state for $(\TT^4,g,\ph)$.  

We define the decomposition $(\PPP,\GGG, \SSS)$ as in Lemma \ref{eqn:theta-r-bv}.  In Lemma \ref{lem:GMspec}, we showed that
$\GGG^M$ has tail specification at scale $3\rho'$, so condition \ref{c.spec} of Theorem \ref{t.generalBV} holds. In Lemma \ref{bowenpropBV}, we showed that $\ph$ has the Bowen property on $\GGG$ at scale $\sixtythree \rho'$, so condition \ref{c.bowen} of Theorem \ref{t.generalBV} holds. We have $P(\PPP \cup \SSS, \ph, 6\eta)= \max \{ P(\PPP, \ph, 6\eta),  P(\SSS, \ph, 6\eta)\}$ and both collections satisfy the hypotheses of Theorem \ref{coreestimateallscales}, and thus we have the upper bound
$$
(1-r) \sup_{x\in Q}\ph(x) + r( \sup_{x\in {\TT^4}}\ph(x) + h + \log L) +H(2r),
$$
and $r$ can be chosen arbitrarily close to $\gamma$. By Lemma \ref{hexp}, $h_g^*(6\eta)<6 \log \lambda$, so by Theorem \ref{coreestimateallscales}, $P(\PPP \cup \SSS, \ph)$ is bounded above by
\[
6 \log \lambda+(1-r) \sup_{x\in Q}\ph(x) + r( \sup_{x\in \TT^4}\ph(x) + h + \log L) + H(2r).
\]
Thus, the hypothesis of Theorem \ref{t.BV} gives that 
\[
P(\PPP \cup \SSS, \ph) + \Var(\ph,\sixtythree\rho') < P(\varphi; g),
\]
which verifies condition \ref{c.gap} of Theorem \ref{t.generalBV}. Finally,  by Theorem \ref{expansivityestimate} and Lemma \ref{BV-satisfieshyp}, we have $\Pexp(\ph,\sixtythree\rho') \leq P(\PPP\cup \SSS,\ph) < P(\ph; g)$.

Combining these ingredients, we see that under the conditions of Theorem \ref{t.BV}, 
all the hypotheses of Theorem \ref{t.generalBV} are satisfied for the decomposition $(\PPP, \GGG, \SSS)$. 
This completes the proof of Theorem \ref{t.BV}.

\section{SRB measures and proof of theorem \ref{main3}}\label{s.srb}
An \emph{SRB measure} for a $C^2$ diffeomorphism $f$ is an ergodic invariant measure $\mu$ that is hyperbolic (non-zero Lyapunov exponents) and has absolutely continuous conditional measures on unstable manifolds \cite[Chapter 13]{BP07}. We assume that $g$ is a $C^2$ diffeomorphism in a $C^1$ neighborhood of a Bonatti--Viana diffeomorphism $\fbv \in \UUU_{\lambda, \rho}$ with $\log \lambda$ and $\rho$ not too large. Explicit bounds required on the parameters for $\fbv$ are given at \eqref{srbgap}.

\subsection{Geometric potential}

As we will see, the potential 
$\phigeo(x):= -\log \mathrm{det}(Dg|_{E^{cu}}(x))$ 
has the property that its unique equilibrium state is the physical SRB measure; the potential with this property is often referred to as the \emph{geometric potential} \cite{IT10,GS14}.}
It is a folklore result that a $C^2$ diffeomorphism with a dominated splitting has H\"older continuous distributions, so that the geometric potential is H\"older continuous. However, to the best of our knowledge a proof has never appeared in the literature. For diffeomorphisms of surfaces, this result is given in \cite{PS2009}. The idea of proof for the general result is to modify the $C^r$ section theorem from Hirsch, Pugh and Shub \cite{HPS}. 
In the appendix, we give a direct proof that $\phigeo$ has the Bowen property on $\GGG$, without using (or showing) H\"older continuity of the distribution.

\subsection{Non-negativity of pressure}\label{sec:nonneg} 

We prove a general result on non-negativity of pressure for the geometric potential associated to an invariant foliation. Let $M$ be a compact Riemannian manifold and $W$ be a $C^0$ foliation of $M$ with $C^1$ leaves.  Suppose there is $\delta>0$ such that
\begin{equation}\label{eqn:bdd-leaf-volume}
\sup_{x\in M} m_{W(x)}(W_\delta(x)) < \infty,
\end{equation}
where $m_{W(x)}$ denotes volume on the leaf $W(x)$ with the induced metric.


\begin{lem}\label{lem:Pgeq0}
Let $W$ be a foliation of $M$ as above, with $\delta>0$ such that \eqref{eqn:bdd-leaf-volume} holds.  Let $f\colon M\to M$ be a diffeomorphism and let $\psi(x) = -\log|\det Df(x)|_{T_x W(x)}|$.  Then $P(\psi; f)\geq 0$.
\end{lem}
\begin{proof}
Note that $\psi$ is continuous because $f$ is $C^1$ and $W$ is $C^0$.  Thus for every $\eps>0$, there is $\delta>0$ such that $d(x,y)<\delta$ implies
\begin{equation}\label{eqn:Ju-close}
|\psi(x) - \psi(y)| < \eps.
\end{equation}
Decreasing $\delta$ if necessary, we can assume that \eqref{eqn:bdd-leaf-volume} holds.  Now for every $x\in M$ and every $y\in B_n(x,\delta)$, we have
\begin{equation}\label{eqn:SnJu-close}
|\det Df^n(y)|_{T_y W(y)}| \geq e^{-\eps n} e^{-S_n\psi(x)}.
\end{equation}
Writing $B_n^W(x,\delta)$ for the connected component of $W(x) \cap B_n(x,\delta)$ containing $x$, we get
\[
m_{W(f^nx)}(f^n B_n^W(x,\delta)) \geq e^{-\eps n} e^{-S_n\psi(x)} m_{W(x)} B_n^W(x,\delta).
\]
Since $f^n B_n^W(x,\delta) \subset W_\delta(f^nx)$, we write $C$ for the quantity in \eqref{eqn:bdd-leaf-volume} and get
$m_{W(x)} B_n^W(x,\delta) \leq Ce^{\eps n} e^{S_n\psi(x)}$ for every $x,n$.

Now let $V$ be a local leaf of $W$.  Given $n\in \NN$, let $Z_n$ be a maximal $(n,\delta)$-separated subset of $V$.  Then $V \subset \bigcup_{x\in Z_n} B_n^W(x,\delta)$, and so 
\[
m_V(V) \leq \sum_{x\in Z_n} m_V B_n^W(x,\delta) \leq \sum_{x\in Z_n} C e^{\eps n} e^{S_n\psi(x)} \leq C e^{\eps n} \Lsep_n(\psi,\delta).
\]
We conclude that $P(\psi; f) \geq P(\psi, \delta; f) \geq -\eps$, and since $\eps>0$ was arbitrary this shows that $P(\psi; f) \geq 0$.
\end{proof}
We claim that Property \eqref{eqn:bdd-leaf-volume} holds for the center-unstable foliation $W^{cu}$ of $g$.  Indeed, each local leaf $W_\delta(x)$ is the graph of a function $\psi\colon F^u\to F^s$ with $\|D\psi\| \leq \beta$, and writing $W_\delta'(x) \subset F^u$ for the projection of $W_\delta(x)$ to $F^u$ along $F^s$, we see that
\begin{enumerate}
\item $W_\delta(x) = (\Id + \psi)(W_\delta'(x))$,
\item $W_\delta'(x)$ is contained inside a ball of radius $\delta(1+\beta)$ in $F^u$, and 
\item $m_{W(x)} W_\delta(x) \leq (1+\|D\psi\|) m_{F^u} W_\delta'(x) \leq (1+\beta) \pi (\delta(1+\beta))^2$.
\end{enumerate}
Thus, we conclude that $P( \phigeo; g)\geq 0$.


\subsection{Negativity of $\Phi(\phigeo; g)$} \label{srbnegbadpressure}
We show that $\Phi(\phigeo; g) < 0$ as long as the parameters in the Bonatti--Viana construction are chosen small. 

Observe that $\sup_{x\in \TT^4}\phigeo(x) \approx \log \lambda - \log\lambda_4$ and $\inf_{x\in \TT^4} \phigeo(x) \approx -(\log\lambda_3 + \log \lambda_4)$. More precisely, given $\epsilon > 0$, we can choose $g$ in a sufficiently small $C^1$ neighbourhood of $\fbv$ 
so that $\sup_{x\in \TT^4}\phigeo(x) \leq \log \lambda - \log\lambda_4 + \epsilon$, and $\inf_{x\in \TT^4} \phigeo(x) \geq -(\log\lambda_3 + \log \lambda_4) - \epsilon$. Thus,
\begin{align*}
\sup\phigeo + \Var(\phigeo,300\rho')  &\leq 2\sup \phigeo - \inf\phigeo \\   &\leq 2\log\lambda + \log\lambda_3 - \log\lambda_4 + 2 \epsilon.
\end{align*}
Thus, we have
\begin{align*}
\Phi(\phigeo;g) &\leq 6\log \lambda + \sup \phigeo  + \gamma( \log L + h) +H(\gamma)  + V \\
& \leq (\log\lambda_3 - \log\lambda_4) + 8\log \lambda + \gamma( \log L + h) +H(\gamma) + 2 \epsilon,
\end{align*}
where, since $\lambda_4> \lambda_3 >1$, the first term is a negative number, and the other terms can be made small. Thus, $\Phi(\phigeo;g) <0$. To be more precise, if $\lambda(\fbv)$ is chosen small enough so that
\begin{equation}\label{srbgap}
8\log \lambda + \gamma( \log L + h) +H(\gamma) < \log \lambda_4 - \log \lambda_3,
\end{equation}
then a sufficiently small $C^1$ perturbation of $\fbv$ satisfies  $\Phi(\phigeo_g;g)<0$. 

Since $\Phi(\phigeo; g) < 0 \leq P(\phigeo; g)$, we can apply Theorem \ref{main2}, and we obtain that $\phigeo$ has a unique equilibrium state.

\subsection{Proof that $\Phi(t\phigeo; g) < P(t\phigeo; g)$ for $t\in[0,1]$} \label{pressuregaptphigeo}
We show that the pressure bound $\Phi(t \phigeo; g) < P( t\phigeo; g)$ for all $t\in [0,1]$ as long as \eqref{srbgap} holds. Since the equality is strict, it will persist for all $t$ in a neighborhood of $[0,1]$. We give linear bounds for $P(t \phigeo; g)$ and  $\Phi(t\phigeo;g)$. First observe that, by the variational principle, 
\begin{align*}
P(t\phigeo; g) &\geq \htop(g) + t \inf \phigeo \\
& \geq \htop(g)- t(\log\lambda_3 + \log \lambda_4+ \eps)  
\end{align*}
Since there is a semi-conjugacy between $g$ and $\fB$, $\htop(g) \geq \htop(\fB) = \log\lambda_3 + \log \lambda_4$. 
Thus, letting $a_1 = \log\lambda_3 + \log \lambda_4$, and
\[
l_1(t) = a_1-t(a_1+ \epsilon),
\]
we have $P(t\phigeo; g) \geq l_1(t)$ and $l_1(t) \geq 0$ whenever $t \leq \frac{a_1}{a_1+ \epsilon}$.

Now, for $\Phi(t\phigeo; g)$, the argument of \S \ref{srbnegbadpressure} shows that
\[
\Phi(t \phigeo;g) \leq  t(\log\lambda_3 - \log\lambda_4  + 2 \epsilon) + 8\log \lambda + \gamma( \log L + h) +H(\gamma).
\]
Thus, letting $a_2 = \log \lambda_4 - \log \lambda_3$ and $r= 8\log \lambda + \gamma( \log L + h) +H(\gamma)$, and 
\[
l_2(t) = r -t(a_2-2\epsilon),
\]
we have $\Phi(t\phigeo ; g) \leq l_2(t)$, and the root of $l_2(t)$ is $ t^\ast =\frac{r}{a_2-2\epsilon}$. Now suppose that 
\begin{equation} \label{srbcriteria2}
\frac{r}{a_2-2\epsilon} < \frac{a_1}{a_1+ \epsilon},
\end{equation}
and that $r<a_1$. This is clearly possible since $r$ can be chosen small. These criteria hold for $\epsilon$ small if \eqref{srbgap} holds for $\lambda= \lambda(\fbv)$. Since $l_2(0)< l_1(0)$ and $l_2(t^\ast)=0<l_1(t^\ast)$, then for $t\in[0, t^\ast]$,
\[
\Phi(t\phigeo; g) \leq l_2(t) < l_1(t) \leq P(t \phigeo).
\]
For $t\in (t^\ast, 1]$, we have $\Phi(t\phigeo; g) \leq l_2(t)<0 \leq P(\phigeo) \leq P(t \phigeo)$. The last inequality holds because since $\sup \phigeo< 0$, the function $t\mapsto P(t\phigeo)$ is decreasing.

We conclude that $\Phi(t \phigeo; g) < P(t \phigeo; g)$ for all $t\in [0,1]$, and thus there exists $\epsilon >0$ so $\Phi(t \phigeo; g) < P(t \phigeo; g)$ for all $t\in [-\epsilon,1+\epsilon]$. We apply Theorem 
\ref{t.BV} to these potentials, and we obtain uniqueness of these equilibrium states, which proves (2) of Theorem \ref{main3}.

\subsection{The formula $P( \phigeo; g) =0$ and $\mu_1$ as SRB measure} \label{sec:lyapunov-exponents-for-bonatti-viana-examples}

Given a $C^2$ diffeomorphism $f$ on a $d$-dimensional manifold and $\mu \in \mathcal{M}_e(f)$, let $\lambda_1 < \cdots < \lambda_s$ be the Lyapunov exponents of $\mu$, and let $d_i$ be the multiplicity of $\lambda_i$, so that $d_i = \dim E_i$, where for a Lyapunov regular point $x$ for $\mu$ we have 
\[
E_i(x) = \{0\} \cup \{ v\in T_xM \, :\,  \lim_{n\to\pm \infty} \tfrac 1n \log \|Df^n_x(v)\| = \lambda_i \} \subset T_x M.
\]
Let $k=k(\mu) = \max \{1\leq i\leq s(\mu) \, :\,  \lambda_i \leq 0\}$, and let $\lambda^+(\mu) = \sum_{i>k} d_i(\mu) \lambda_i(\mu)$ be the sum of the positive Lyapunov exponents, counted with multiplicity.

The Margulis--Ruelle inequality  \cite[Theorem 10.2.1]{BP07} gives $h_\mu(f) \leq \lambda^+(\mu)$, and it was shown by Ledrappier and Young \cite{LY} that equality holds if and only if $\mu$ has absolutely continuous conditionals on unstable manifolds. Thus, for any ergodic invariant measure $\mu$, we have
\begin{equation}\label{eqn:nonpos}
h_\mu(f) - \lambda^+(\mu)\leq 0,
\end{equation}
with equality if and only if $\mu$ is absolutely continuous on unstable manifolds.  In conclusion, an ergodic measure $\mu$ is an SRB measure if and only if it is hyperbolic and equality holds in \eqref{eqn:nonpos}.

In this section, we prove that $P(\phigeo; g) \leq 0$.  Combining this with Lemma \ref{lem:Pgeq0} gives that $P(\phigeo ; g) =  0$. 
Recall that in the previous section we used Theorem \ref{t.BV} to show that $\phigeo$ has a unique equilibrium state $\mu$; to show that 
$\mu$ is the SRB measure, we need to show that $\mu$ is hyperbolic and $\lambda^+(\mu) = \int\phigeo\,d\mu$. 

\subsubsection*{Lyapunov exponents for the diffeormorphism g}

Let $\mu$ be ergodic, and let $\lambda_1(\mu) \leq \lambda_2(\mu) \leq \lambda_3(\mu) \leq \lambda_4(\mu)$ be the Lyapunov exponents for $\mu$.
Recall that $E^{cs} \oplus E^{cu}$ is 
$Dg$-invariant, so for every $\mu$-regular $x$ the Oseledets decomposition is a sub-splitting of $E^{cs} \oplus E^{cu}$.
\begin{lem} \label{BV:Ruelle}
For an ergodic measure $\mu$, then
\begin{equation} \label{eqn:phlambda2} 
\int \phigeo\,d\mu \geq - \lambda^+(\mu).
\end{equation}
\end{lem}
\begin{proof} 
Because $E^{cs} \oplus E^{cu}$ is dominated, standard arguments show that $\int \phigeo\,d\mu = -\lambda_3(\mu) - \lambda_4(\mu)$.  There are three cases.
\begin{enumerate}
\item
If $\mu$ has exactly two positive Lyapunov exponents (counted with multiplicity), then $\int \phigeo\,d\mu =- \lambda^+(\mu)$.
\item
If $\lambda_2(\mu)\geq0$, then $\int \phigeo\,d\mu \geq -\lambda_2(\mu) - \lambda_3(\mu) - \lambda_4(\mu) \geq - \lambda^+(\mu) $.
\item
There is at most one positive Lyapunov exponent. In this case, $-\lambda_3\geq0$, so $\int \phigeo\,d\mu \geq -\lambda_4(\mu) \geq - \lambda^+(\mu)$. \qedhere  
\end{enumerate}
\end{proof}

Let $\mathcal{M}_* \subset \mathcal{M}_e(g)$ be the set of ergodic $\mu$ such that $\mu$ is hyperbolic  and has exactly two positive exponents, so $\lambda_2(\mu) < 0 < \lambda_3(\mu)$.  
\begin{lem} \label{lem:keySRBestimate2}
If $\mu \in \mathcal{M}_e(g) \setminus \mathcal{M}_*$, then
\[
h_\mu(g) -\lambda^+(\mu) \leq h_\mu(g) + \int\phigeo d \mu  \leq \Phi(\phigeo; g) 
\]
\end{lem}
\begin{proof}
The first inequality follows from Lemma \ref{BV:Ruelle}, so our work is to prove the second.
Suppose that $\mu \in \mathcal{M}_e(g) \setminus \mathcal{M}_*$, and that either $\mu$ belongs to Case (1) and is not hyperbolic, or belongs to Case (2) in the proof of Lemma \ref{BV:Ruelle}. Then there exists a set $Z\subset M$ with $\mu(Z)=1$ so that for each $z \in Z$, there exists $v \in E^{cs}_z$ with
$
\lim_{n\to\infty} \tfrac 1n \log \|Dg^n_z(v)\| \geq 0.
$ 
Thus with $r>\gamma$, we have $z\in A^+$, where as in \eqref{eqn:A+} we put
\begin{equation}\label{eqn:A++}
A^+=\{x : \text{there exists } K(x) \text{ so } \tfrac{1}{n}S^g_{n}\chi(x) < r \text{ for all } n>K(x)\}.
\end{equation}
To see this, suppose that $z \notin A^+$. Then there exists $n_k \to \infty$ with $\frac{1}{n_k}S^g_{n_k}\chi(z) \geq \gamma$. By Lemma \ref{unifexpansion}, this gives 
\[
 \|Dg^{n_k}_z(v)\| \leq \|Dg^{n_k}|_{E^{cs}(z)}\| \leq (\theta_r)^{n_k},
\]
and thus 
$
\lim_{n_k\to\infty} \tfrac 1{n_k} \log \|Dg^{n_k}_z(v)\| \leq \log \theta_r<0,
$
which is a contradiction. Thus, $\mu(A^+) =1$, where $A^+$ is as in \eqref{eqn:A++}.Writing $\CCC = \CCC(q,r;g)$ and $\CCC'=\CCC(q',r;g)$ (where the notation is defined in \eqref{eqn:C}), it follows that
\[
h_\mu(g) -\lambda^+(\mu) \leq h_\mu(g) + \int\phigeo\,d\mu \leq P(\CCC \cup \CCC',\phigeo) \leq \Phi(\phigeo; g),
\]
where the first inequality uses \eqref{eqn:phlambda2}, the second uses Lemma \ref{keystepexpansivityestimate}, and the third uses Theorem \ref{coreestimateallscales}.

Now suppose $\mu$ belongs to case (3) above, and thus there is a non-positive exponent associated to $E^{cu}$. An analogous argument shows that $\mu(A^-)>0$, where 
\[
A^-=\{x : \text{there exists } K(x) \text{ so } \tfrac{1}{n}S^{g^{-1}}_{n}\chi(x) < r \text{ for all } n>K(x)\}.
\]
The key point is that there  exists a set $Z\subset M$ with $\mu(Z)=1$ so that for each $z \in Z$, there exists $v \in E^{cu}_z$ with
\[
\lim_{n\to-\infty} \tfrac 1n \log \|Dg^{-n}_z(v)\| \geq 0.
\] 
It follows that $z\in A^-$,  because otherwise there exists $n_k \to \infty$ with $\frac{1}{n_k}S^{g^{-1}}_{n_k}\chi(z) \geq \gamma$, and thus by lemma \ref{unifexpansion}, we have
\[
 \|Dg^{-n_k}_z(v)\| \leq \|Dg^{-n_k}|_{E^{cs}(z)}\| \leq (\theta_r)^{n_k},
\]
and thus $ \lim_{n_k\to -\infty} \tfrac 1{n_k} \log \|Dg^{-n_k}_z(v)\| \leq \log \theta_r<0$, which is a contradiction. Thus, $\mu(A^-) =1$. Again, it follows that
\[
h_\mu(g) -\lambda^+(\mu) \leq h_\mu(g) + \int\phigeo\,d\mu \leq P(\CCC,\phigeo) \leq \Phi(\phigeo; g).
\]
where the first inequality uses \eqref{eqn:phlambda2}, the second uses Lemma \ref{keystepexpansivityestimate}, and the third uses Theorem \ref{coreestimateallscales}. 
\end{proof}

\subsubsection*{Completing the proof}
It follows from \S \ref{srbnegbadpressure}, Lemma \ref{lem:keySRBestimate2} and Lemma \ref{lem:Pgeq0} that any ergodic $\mu$ not in $\mathcal{M}_*$ satisfies
\[
h_\mu(g) + \int\phigeo\,d\mu \leq \Phi(\phigeo) < 0 \leq P(\phigeo).
\]
Thus, it follows from the variational principle that
\begin{equation}\label{eqn:vp*2}
P(\phigeo) = \sup \left\{h_\mu(g) + \int\phigeo\,d\mu \, :\,  \mu \in \mathcal{M}_* \right\}.
\end{equation}
Now, for every $\mu\in \mathcal{M}_*$, we have $\int\phigeo\,d\mu = -\lambda^+(\mu)$, and thus
\begin{equation}\label{eqn:free-energy2}
h_\mu(g) + \int\phigeo\,d\mu = h_\mu(g) -\lambda^+(\mu)\leq 0.
\end{equation}
It follows that $P(\phigeo) = \sup \left\{h_\mu(g) + \int\phigeo\,d\mu \, :\,  \mu \in \mathcal{M}_* \right\} \leq0$. Hence, $P(\phigeo)=0$. Since $\sup \phigeo< 0$, the function $t\mapsto P(t\phigeo)$ is a convex strictly decreasing function from $\RR\to\RR$, and thus $1$ is the unique root. 

To show that the unique equilibrium state $\mu$ is an SRB measure for $g$, we observe that $\mu\in \mathcal{M}_*$ implies that $\mu$ is hyperbolic, and since $P(\phigeo)=0$, \eqref{eqn:free-energy2} gives $h_\mu(g) - \lambda^+(\mu)=0$, so $\mu$ is an SRB measure.

To see that there is no other SRB measure, we observe that if $\nu\neq \mu$ is any ergodic measure, then $h_\nu(g) - \lambda^+(\nu) \leq h_\nu(g) + \int\phigeo\,d\nu < P(\phigeo)=0$ by \eqref{eqn:phlambda2} and the uniqueness of $\mu$ as an equilibrium measure. This completes the proof of Theorem \ref{main3}.

\section{Proofs of Lemmas}\label{s.lemmas}

\begin{proof} [Proof of Lemma \ref{lem:span-sep}]
It suffices to consider $(n,\delta)$-separated sets of maximum cardinality in the supremum for the partition sum. Otherwise, we could increase the partition sum by adding in another point. An  $(n,\delta)$-separated set of maximum cardinality must be $(n,\delta)$-spanning, or else we could add in another point and still be  $(n,\delta)$-separated.  The first inequality follows.

For the second inequality, let $E_n$ be any $(n,2\delta)$-separated set and $F_n$ any $(n,\delta)$-spanning set. Define the map $\pi\colon E_n\to F_n$ by choosing for each $x\in E_n$ a point $\pi(x)$ with the property that $d(x,\pi(x))\leq \delta$. The map $\pi$ is injective. Thus, for any $E$ which is $(n, 2\delta)$ separated, 
\[
\sum_{y\in F_n} e^{S_n\ph(y)} \geq \sum_{x \in E_n} e^{S_n \ph (\pi(x))} \geq \sum_{x \in E_n} e^{S_n \ph (x) - n \Var(\ph, \delta)},
\]
and thus $\sum_{y\in F_n} e^{S_n\ph(y)} \geq e^{- n \Var(\ph, \delta)}\Lsep_n(\DDD,\ph,2\delta)$.
\end{proof}

\begin{proof}[Proof of Lemma \ref{smallerscales}]
It is shown in \cite[Proposition 2.2]{rB72} that given any $\delta>0$ and $\alpha>h^*_f(\eps)$, there is a constant $K$ such that
\[
\Lspan(B_n(x,\eps),0,\delta;f) \leq K e^{\alpha n}
\]
for every $x\in X$ and $n\in \NN$; that is, every Bowen ball $B_n(x,\eps)$ has an $(n,\delta)$-spanning subset $F_{x,n}$ with cardinality at most $K e^{\alpha n}$.  Let $E_n \subset \DDD_n$ be a maximal $(n, \epsilon)$-separated set. Then 
$G_n = \bigcup_{x\in E_n} F_{x,n}$ is $(n,\delta)$-spanning for $\DDD_n$, and has
\[
\sum_{y\in G_n} e^{S_n\ph(y)}
\leq \sum_{x\in E_n} e^{S_n\ph(x)} e^{n\Var(\ph,\eps)} K e^{\alpha n}.
\]
We conclude that
$
\Lspan_n(\DDD,\ph,\delta) \leq \Lsep_n(\DDD,\ph,\eps) K e^{n(\Var(\ph,\eps) + \alpha)}.
$
Then the second inequality in Lemma \ref{lem:span-sep} gives
\[
\Lsep_n(\DDD,\ph,2\delta) \leq e^{n\Var(\ph,\delta)} \Lsep_n(\DDD,\ph,\eps) K e^{n(\Var(\ph,\eps) + \alpha)};
\]
sending $n\to\infty$ gives the first half of Lemma \ref{smallerscales}, and sending $\delta\to 0$ gives the second half.
\end{proof}


\begin{proof}[Proof of Proposition \ref{pressuredrop}.] 
With $\eta$ and $C$ as in the statement of the lemma, put $\alpha=\eta/C$.  By the Anosov shadowing lemma if $\{x_n\}$ is an $\alpha$-pseudo orbit for $f$, then there exists an $f$-orbit that $\eta$-shadows $\{x_n\}$.

Now fix $g\in \mathrm{Diff}(M)$ with $d_{C^0}(f, g)<\alpha$.  Then every $g$-orbit is an $\alpha$-pseudo orbit for $f$, and hence for every $x\in M$, we can find a unique point $\pi(x)\in M$ such that
\begin{equation}\label{eqn:fhg}
d(f^n(\pi x), g^n x) < \eta \text{ for all }n\in \mathbb{Z}.
\end{equation}
We prove \ref{Pg-geq}. By expansivity of $f$, we have
\begin{equation}\label{eqn:Pf}
P(\ph; f) = \lim_{n\to\infty} \frac 1n \log \Lspan_n(\ph,3\eta;f).
\end{equation}

Let $E_n$ be a $(n,\eta)$-spanning set for $g$.  Then from \eqref{eqn:fhg} we see that $\pi(E_n)$ is $(n,3\eta)$-spanning for $f$. 
It follows that
\begin{equation}\label{eqn:Lf}
\Lspan_n(\ph,2\eta;f) \leq \sum_{x\in \pi(E_n)} e^{S_n^{f}\ph(x)}
= \sum_{x\in E_n} e^{S_n^{f}\ph(\pi x)}.
\end{equation}
Note that  $S_n^f\ph(\pi x)= \sum_{k=0}^{n-1} \ph(f^k(\pi x))
\leq \sum_{k=0}^{n-1} (\ph(g^kx) + \Var (\varphi, \eta))$,
and together with \eqref{eqn:Pf} and \eqref{eqn:Lf} this gives
\[
P(\ph; f) \leq \lim_{n\to\infty} \frac 1n \log \sum_{x\in E_n} e^{n \Var (\varphi, \eta)+ S_n^g\ph(x)}.
\]
Taking an infimum over all $(n,\eta)$-spanning sets for $g$ gives 
\[
P(\ph; f) \leq \Var (\varphi, \eta) + P(\ph,\eta; g)
\]
by the first inequality in Lemma \ref{lem:span-sep}.  This  completes the proof of \ref{Pg-geq} since $P(\ph; g) \geq P(\ph,\eta; g)$.

Now we prove \ref{Lambdag-leq}. Let $E_n$ be a maximal $(n, 3\eta)$ separated set for $g$.  As in the previous argument, we see from \eqref{eqn:fhg} that $\pi(E_n)$ is $(n,\eta)$-separated for $f$:
indeed, for every $x,y\in E_n$ there is $0\leq k<n$ such that $d(g^kx,g^ky) \geq 3\eta$, and hence
\[
d(f^k(\pi x),f^k(\pi y)) \geq d(g^kx,g^ky) - d(g^kx, f^k\pi x) - d(g^ky, f^k\pi y) > \eta.
\]
In particular, we have
\[
\begin{aligned}
\Lsep_n(\ph,\eta; f) &\geq \sum_{x\in \pi(E_n)} e^{S_n^{f}\ph(x)} 
= \sum_{x\in E_n} e^{S_n^{f}\ph(\pi x)}  \\
&\geq \sum_{x\in E_n} e^{S_n^{g}\ph(x)-n\Var(\phi, \eta)} \geq \Lsep_n(\ph, 3\eta; g) e^{-n\Var(\ph, \eta)}. \qedhere
\end{aligned}
\]
\end{proof}



\begin{proof}[Proof of Lemma \ref{lem:Wlps}]
Given $x,y\in F^1 \oplus F^2$, let $z'$ be the unique point of intersection of $(x+F^1) \cap (y+F^2)$.  Translating the coordinate system so that $z'$ becomes the origin, we assume w.l.o.g. that $x\in F^1$ and $y\in F^2$.  Then $W^1(x)$ and $W^2(y)$ are graphs of $C^1$ functions $\phi_1\colon F^1\to F^2$ and $\phi_2\colon F^2\to F^1$ with $\|D\phi_i\| < \beta$.
That is, $W^1(x) = \{a + \phi_1(a) \, :\,  a\in F^1\}$ and $W^2(y) = \{\phi_2(b) +b \, :\,  b\in F^2\}$.  Thus $z\in W^1 \cap W^2$ if and only if $z=a+\phi_1(a) = \phi_2(b)+b$ for some $a\in F^1$ and $b\in F^2$.  This occurs if and only if $b=\phi_1(a)$ and $a=\phi_2(b)$; that is, if and only if $a = \phi_2\circ \phi_1(a)$ and $b=\phi_1(a)$.  Because $\phi_2\circ \phi_1$ is a contraction on the complete metric space $F^1$ it has a unique fixed point $a$.

For the estimate on the distances from $z$ to $x,y$ we observe that
\begin{align*}
\|a\| &= d(a,0) = d(\phi_2b,\phi_2y) \leq \beta d(b,y) \leq \beta (\|b\| + \|y\|), \\
\|b\| &= d(b,0) = d(\phi_1a,\phi_1x) \leq \beta d(a,x) \leq \beta (\|a\| + \|x\|).
\end{align*}
Recall that by the definition of $\bk$ we have $\|x\|,\|y\| \leq \bk\|x-y\|$.  Thus we have
\[
\|a\| \leq \beta(\beta(\|a\| + \|x\|) + \|y\|)
\leq \beta^2\|a\| + \beta(1+\beta)\bk d(x,y),
\]
which gives $\|a\| \leq \frac{\beta}{1-\beta} \bk d(x,y)$, and similarly for $\|b\|$.  Thus
\[
d(a,x) \leq \|a\| + \|x\| \leq \left(\frac{\beta}{1-\beta} + 1\right)\bk d(x,y) = \frac {\bk d(x,y)}{1-\beta}.
\]
To obtain the bound on $d_{W^1}(z,x)$, observe that there is a path $\gamma$ from $a$ to $x$ with length $\leq \frac{\bar\kappa}{1-\beta}d(x,y)$; the image of $\gamma$ under the map $\Id + \phi_1$ connects $z$ to $x$ and has length $\leq \frac{1+\beta}{1-\beta}\bar\kappa d(x,y)$ since $\|\Id+\phi_1\| \leq 1+\beta$.
The other distance bound is similar. \end{proof}

\begin{proof}[Proof of Lemma \ref{compare:dist}] 
Suppose $W=W^1$; the case $W=W^2$ is similar.  Let $y'$ be the intersection point of $y+F_2$ and $x+F_1$.  Then since $x,y$ lie on the same leaf of $W^1$, we must have $y-x \in C_\beta(F^1,F^2)$, and so $\|y-y'\| / \|y'-x\| \leq \beta$.  This gives
\[
\|y-x\| \geq \|y'-x\| - \|y-y'\| \geq \|y'-x\|(1-\beta),
\]
so $\|y'-x\| \leq (1-\beta)^{-1} \|y-x\|$.
Now $W^1(x)$ is the image of $x+F^1$ under a map $G$ with $\|DG\| \leq 1+\beta$, so there is a curve on $W^1(x)$ connecting $x$ and $y$ with length $\leq (1+\beta) \|y' - x\|$.  This completes the proof.
\end{proof}

\begin{proof}[Proof of Lemma \ref{lem:WcuWcs}]
We  use the following general lemma.

\begin{lem}\label{lem:leaves-dense}
Let $W$ be a foliation of a compact manifold $M$ such that $W(x)$ is dense in $M$ for every $x\in M$.  Then for every $\alpha>0$ there is $R>0$ such that $W_R(x)$ is $\alpha$-dense in $M$ for every $x\in M$.
\end{lem}
\begin{proof}
Given $R>0$, define a function $\psi_R\colon M\times M \to [0,\infty)$ by $\psi_R(x,y) = \mathrm{dist}(y,W_R(x))$.  Note that for each $R$, the map $x\mapsto W_R(x)$ is continuous (in the Hausdorff metric) and hence $\psi_R$ is continuous.  Moreover, since $W(x) = \bigcup_{R>0} W_R(x)$ is dense in $M$ for each $x\in M$, we have $\lim_{R\to\infty} \psi_R(x,y) = 0$ for each $x,y\in M$.  Finally, when $R\geq R'$ we see that $W_R(x) \supset W_{R'}(x)$ and so $\psi_R(x,y) \leq \psi_{R'}(x,y)$.  Thus $\{\psi_R\, :\,  R>0\}$ is a family of continuous functions that converge monotonically to 0 pointwise.  By compactness of $M\times M$, the convergence is uniform, hence for every $\alpha>0$ there is $R$ such that $\psi_R(x,y) <\alpha$ for all $x,y\in M$.
\end{proof}
Now put $\delta=\rho'$.  By the local product structure for $W^{cs},W^u$ we can put $\alpha = \delta/\kappa $ and observe that if $d(y,z)<\alpha$, then $W^u_\delta(z) \cap W^{cs}_\delta(y) \neq \emptyset$.
By Lemma \ref{lem:leaves-dense}, there is $R>0$ such that $W_R^u(x)$ is $\alpha$-dense in $\TT^d$ for every $x\in \TT^d$.  Thus for every $x\in \TT^d$ there is $z\in W_R^u(x)$ such that $d(y,z)<\alpha$, and thus $W^u_\delta(z) \cap W^{cs}_\delta(y) \neq \emptyset$.  The result follows by observing that $W_{R+\delta}^u(x) \supset W_\delta^u(z)$.
\end{proof}

\appendix
\section{The geometric potential has the Bowen property on $\mathcal{G}$}

In this appendix, we give a a direct proof that the geometric potential $\phigeo:= -\log \mathrm{det}(Dg|_{E^{cu}})$ has the Bowen property on $\GGG$ {when $g$ is $C^{1+\alpha}$}. This allows us to treat scalar multiples of the geometric potential using Theorem \ref{t.BV} without relying on the folklore result that a $C^2$ diffeomorphism with a dominated splitting has H\"older continuous distributions.  
One advantage of this approach is that the argument is suitable for generalization to non-uniformly hyperbolic settings, where H\"older continuity may fail. The main idea is Lemma \ref{lem:Grassmann-contract} below, which gives contraction estimates for the action of $Dg$ on the Grassmannian. 

\subsection{Action on the Grassmannian}

The standard approach to the geometric potential in the uniformly hyperbolic case 
is to argue that the unstable distribution is H\"older continuous (i.e. the section $x \mapsto E^{u}(x)$ is H\"older continuous), and thus the map $\phigeo(x)=-\log \mathrm{det}(Dg|_{E^{u}})(x)$ is H\"older. This approach is captured on the following commutative diagram:
\[
\xymatrix{
& G \ar[dr]^{\psi} & \\
M \ar[ur]^{E^{u}} 
\ar[rr]_{\phigeo} && \RR
}
\]
where $G$ is the appropriate Grassmannian bundle over $M$, and $\psi$ sends $E\in G$ to $-\log |\det Dg(x)|_E|$. Note that all we need for $\psi$ to be H\"older continuous is for the map $g$ to be $C^{1+\alpha}$ (see Lemma \ref{lem:Grass-Holder} below). Thus, the question of regularity of $\phigeo$ reduces to the question of regularity for $E^u\colon M \to G$.

In our setting, where $\phigeo(x) =-\log \det(Df|_{E^{cu}(x)})$
we obtain refined estimates on $E^{cu}\colon \TT^4 \to G$ for good orbit segments, which allow us to establish the Bowen property on these segments.

More precisely, we let $G_2$ denote the Grassmannian bundle of $2$-planes in $\RR^4$ over the torus. Since the underlying manifold is the torus, this is a product bundle, and we can identify $G_2$ with $\TT^4 \times \text{Gr}(2, \RR^4)$, where $\text{Gr}(2, \RR^4)$ is the space of planes through the origin in $\RR^4$. The map $g$ induces dynamics on $G_2$ by the formula
\begin{equation}\label{eqn:gDg}
(x, V) \mapsto (g(x), Dg(V)).
\end{equation}
We show here that $\psi$ is H\"older, and in \S\ref{sec:reduction} that it suffices to prove the Bowen property for trajectories that start on the stable manifold of $x$; then in \S\ref{sec:grass-contract} we do this by studying the dynamics of \eqref{eqn:gDg}.

Note that $\text{Gr}(2,\RR^4)$ is equipped with the metric
\[
d_G(E,E') = d_H (E\cap S^3, E'\cap S^3),
\]
where $d_H$ is the Hausdorff metric on compact subsets of the unit sphere $S^3 \subset \RR^4$.  We will use the fact that on small neighborhoods $U\subset \text{Gr}(2,\RR^4)$ one can define a Lipschitz map $U\to \RR^4\times \RR^4$ that assigns to each $E\in U$ an orthonormal basis for $E$.

\begin{lem}\label{lem:Grass-Holder}
If $g\colon \TT^4\to \TT^4$ is $C^{1+\alpha}$, then the map $\psi\colon \TT^4\times \text{Gr}(2,\RR^4) \to \RR$ given by $\psi(x,E) = -\log|\det Dg(x)|_E|$ is H\"older continuous with exponent $\alpha$.
\end{lem}
\begin{proof}
Given $v,w\in \RR^4$, the square of the area of the parallelogram spanned by $v,w$ is given by the smooth function
$A(v,w) = \sum_\sigma v_{\sigma(1)} w_{\sigma(2)}$, where the sum is over all 1-1 maps $\sigma\colon \{1,2\}\to \{1,2,3,4\}$.  Given $(x,E) \in \TT^4\times \text{Gr}(2,\RR^4)$, let $v,w$ be an orthonormal basis for $E$, so
\[
\psi(x,E) = -\frac 12\log \left| \frac{A(Dg_x(v), Dg_x(w))}{A(v,w)}\right|.
\]
The function $Dg$ is $\alpha$-H\"older, the function $\log$ is Lipschitz on compact subsets of $(0,\infty)$, and $\|Dg^{\pm 1}\|$ is bounded away from $0$ and $\infty$, and $A$ is smooth, so we conclude that $\psi$ is $\alpha$-H\"older.
\end{proof}

\subsection{Reduction to the centre-stable manifold}\label{sec:reduction}

In this section and the next we prove the following result, which together with Lemma \ref{bowen-ballsBV} and Lemma \ref{lem:Grass-Holder} implies the Bowen property for $\phigeo$ by following the same computation as in Lemma \ref{bowenpropBV}.

\begin{lem}\label{lem:Grassmann-contract}
There are $C\in \RR$ and $\theta<1$ such that for every $(x,n)\in \GGG$, $y\in B_{\sixtythree \rho'}(x,n)$, and $0\leq k\leq n$, we have
\[
d_H(E^{cu}(g^kx), E^{cu}(g^ky)) \leq C(\theta^k + \theta^{n-k}).
\]
\end{lem}

Note that here we identify both $E^{cu}(g^kx)$ and $E^{cu}(g^ky)$ with subspaces of $\RR^4$, and Lemma \ref{lem:Grassmann-contract} gives a bound on the distance between these subspaces; the corresponding bound on the distance between $g^kx$ and $g^ky$ was already proved in Lemma \ref{bowen-ballsBV}.

The first step in the proof of Lemma \ref{lem:Grassmann-contract} is exactly as in Lemma \ref{bowen-ballsBV}: 
Using the local product structure at scale $\sixtythree\rho'$,  there exists $z\in W^{cs}_{\kappa \sixtythree\rho'}(x) \cap W^{cu}_{\kappa \sixtythree\rho'}(y) = W^{cs}_{\rho''}(x) \cap W^{cu}_{\rho''}(y) $.  Because the leaves of the foliation $W^{cu}$ are $C^1$, there is a constant $C$ such that
\[
d_H(E^{cu}(g^kz),E^{cu}(g^ky)) \leq C d(g^kz,g^ky)
\leq C (\kappa\sixtythree\rho') \theta_r^{n-k},
\]
using the fact that $z\in W^{cu}_{\rho''}(y)$.  Thus in order to prove Lemma \ref{lem:Grassmann-contract}, it suffices to show that
\begin{equation}\label{eqn:same-stable}
d_H(E^{cu}(g^kx),E^{cu}(g^kz)) \leq C \theta^k
\end{equation}
whenever $z\in W_{\rho''}^{cs}(x)$, which we do in the next section.

\subsection{Unstable directions approach each other}\label{sec:grass-contract}

We fix $(x,n)\in \GGG$.  Given $z\in W_{\rho''}^{cs}(x)$ and $0\leq k\leq n$, let $(e_{z,k}^i)_{i=1}^4$ be an orthonormal basis for $T_{g^kz}\TT^4$ such that $E^{cs}(g^k z) = \text{span}(e_{z,k}^1,e_{z,k}^2)$.  Let $\pi_{z,k}\colon T_{g^kz}\TT^4 \to \RR^4$ be the linear map that takes $v\in T_{g^kz}\TT^4$ to its coordinate representation in the basis $e_{z,k}^i$.  
We can choose the vectors $e_{z,k}^i$ in such a way that for every $k,i$, the map $z\mapsto e_{z,k}^i$ is $K$-Lipschitz on $g^k(W_{\rho''}^{cs}(x))$, where $K$ is a constant that does not depend on $(x,n)$.

Now let $A_k^z\colon \RR^4\to\RR^4$ be the coordinate representation of $Dg_{g^kz}$ in the bases chosen above.  That is, $A_k^z$ makes the following diagram commute.
\[
\xymatrixcolsep{4pc}\xymatrix{
T_{g^k z} \TT^4 \ar[r]^{Dg_{g^kz}} \ar[d]^{\pi_{z,k}} & T_{g^{k+1} z}\TT^4 \ar[d]^{\pi_{z,k+1}} \\
\RR^4 \ar[r]^{A_{k}^z} & \RR^4
}
\]
To prove \eqref{eqn:same-stable}, it suffices to consider $\hat E_k^z := \pi_{z,k} E^{cu}(g^kz)$ and show that
\begin{equation}\label{eqn:hat-E}
d_H(\hat E_k^x, \hat E_k^z) \leq C \theta^k,
\end{equation}
since $\pi_{z,k}^{-1}$ is $K$-Lipschitz in $z$ for each $k$.  Since $\hat E_{k+1}^z = A_k^z \hat E_k^z$ and $\hat E_{k+1}^x = A_k^x \hat E_k^x$, we must study the dynamics of $A_k^z$ and $A_k^x$.

Let $Z = \RR^2\times\{0\} \subset \RR^4$ and note that $Z = \pi_{z,k}E^{cs}(g^kz)$ for every $z,k$.  In particular, this means that $A_k^z(Z)=Z$.

Let $\mathcal{E}$ be the collection of all subspaces $E\subset \RR^4$ such that $\RR^4 = Z\oplus E$.  Given $0\leq k\leq n$, let $E_k = \hat E_k^x$, and for each $E\in \mathcal{E}$, let $L_k^E\colon E_k\to Z$ be the linear map whose graph is $E$.

\begin{lem}\label{lem:angle-norm}
Given any $0\leq k\leq n$ and $E\in \mathcal{E}$, we have
\[
\sin(d_G(E,E_k)) \leq \|L_k^E\|.
\]
\end{lem}
\begin{proof}
Given $v\in E_k$, let $\theta=\theta(v)$ be the angle between $v$ and $v+L_k^Ev\in E$.  By the definition of $d_G$, we have $d_G(E,E_k) \leq \sup_v \theta(v)$, so it suffices to show that $\sin\theta \leq \|L_k^E\|$ for all $v$.  Consider the triangle with vertices at $0$, $v$, and $v+L_k^Ev$.  The side opposite $\theta$ has length $\|L_k^Ev\|$, and the side from $0$ to $v$ has length $\|v\|$.  Writing $\beta$ for the angle opposite this side, the law of sines gives
\[
\frac{\sin\theta}{\|L_k^Ev\|} = \frac{\sin\beta}{\|v\|} \leq \frac 1{\|v\|}.
\]
Multiplying both sides by $\|L_k^Ev\|$ gives the result.
\end{proof}

\begin{lem}\label{lem:E-contraction}
Given $0\leq k\leq n$, an invertible linear map $A\colon \RR^4\to \RR^4$ that preserves $Z$, and a subspace $E\in \mathcal{E}$, let $P_0\colon E_{k+1} \to AE_k$ be projection along $Z$.  Then
\begin{equation}\label{eqn:LAE}
L_{k+1}^{AE} + \Id = (A|_Z \circ L_k^E \circ A|_{E_k}^{-1} + \Id)\circ P_0.
\end{equation}
In particular, we have
\begin{equation}\label{eqn:LAE-norm}
\|L_{k+1}^{AE}\| \leq \|A|_Z\| \cdot \|A|_{E_k}^{-1}\| \cdot \|P_0\| \cdot \|L_k^E\| + \|P_0 - \Id\|.
\end{equation}
\end{lem}
\begin{proof}
Given $v\in E_{k+1}$, let $v_0 = P_0v\in AE_k$.  Then we have
\begin{align*}
v_0 &\in A E_k \quad\Rightarrow  & A^{-1}v_0 &\in  E_k,\\
L_{k+1}^{AE}v + v - v_0 &\in Z
\qquad \Rightarrow & A^{-1}(L_{k+1}^{AE}v + v - v_0) &\in Z, \\
v_0 + (L_{k+1}^{AE}v + v - v_0) &\in AE
\quad\Rightarrow &A^{-1}v_0 + A^{-1}(L_{k+1}^{AE}v + v - v_0) &\in E,
\end{align*}
where the implication in the second row uses invariance of $Z$.
By the definition of $L_k^E$, this implies that
\[
A^{-1}(L_{k+1}^{AE}v + v - v_0) = L_k^A A^{-1} v_0.
\]
Since $v_0 \in AE_k$, we can write $A^{-1} v_0 = A|_{E_k}^{-1} v_0$, and since $L_k^A A|_{E_k}^{-1} v_0 \in Z$ we can apply $A|_Z$ to both sides and write
\[
L_{k+1}^{AE} v + v - P_0 v = A|_Z L_k^A A|_{E_k}^{-1} P_0 v.
\]
Adding $P_0v$ to both sides gives \eqref{eqn:LAE}.
\end{proof}

In particular, when $A = A_k^z$ for $z\in W_{\rho''}^{cs}(x)$, we can use the estimate $d(g^k x, g^kz) \leq \rho'' \theta_r^k$ together with H\"older continuity of $Dg$ and Lipschitz continuity of $e_{z,k}^i$ to deduce that
\[
\|A_k^z - A_k^x\| \leq C(\rho'')^{\alpha} \theta_r^{\alpha k},
\]
and hence
\[
d_G(E_{k+1}, A_k^z E_k)
= d_G(A_k^x E_k, A_k^z E_k) \leq C' (\rho'')^{\alpha}\theta_r^{\alpha k}.
\]
Since $\measuredangle(Z,E_k)$ is bounded away from 0, this implies that the map $P_0$ in Lemma \ref{lem:E-contraction} satisfies $\|P_0 - \Id\| \leq C''(\rho'')^{\alpha} \theta_r^{\alpha k}$ when the Lemma is applied with $A = A_k^z$.  We conclude that
\begin{equation}\label{eqn:hat-E-contract}
\|L_{k+1}^{\hat E_{k+1}^z} \|
\leq \|A_k^z|_Z\| \cdot \|A_k^z|_{E_k}^{-1}\| (1+C''(\rho'')^{\alpha}\theta_r^{\alpha k})
\|L_k^{\hat E_k^z}\| + C''(\rho'')^\alpha \theta_r^{\alpha k}.
\end{equation}
Because the splitting is dominated and the cones are small, there is $\lambda<1$ such that $\|A_k^z|_Z\| \cdot \|A_k^z|_{E_k}^{-1}\| \leq \lambda$ for all choices of $x,z,n,k$, and thus writing $D_k = \|L_k^{\hat E_k^z}\|$, we get $d_G(\hat E_k^x, \hat E_k^z) \leq D_k$ from Lemma \ref{lem:angle-norm}, and \eqref{eqn:hat-E-contract} gives
\begin{equation}\label{eqn:Dk}
D_{k+1} \leq \lambda(1+Q\theta^k)D_k + Q\theta^k,
\end{equation}
where $\theta = \theta_r^\alpha$.  Iterating \eqref{eqn:Dk} shows that $D_k$ decays exponentially; indeed, fixing $\nu < 1$ such that $\lambda,\theta < \nu$ and writing $C_k = D_k \nu^{-k}$, we have
\[
C_{k+1} \leq \frac \lambda\nu (1+Q\theta^k) C_k + Q\frac{\theta^k}{\nu^{k+1}},
\]
and by taking $k_0$ large enough that $\xi := \frac \lambda\nu(1+Q\theta^k) < 1$, this gives
\[
C_{k+1} \leq \xi C_k + Q\nu^{-1}(\theta/\nu)^{k_0}
\]
for all $k\geq k_0$, so that in particular if $C_k \leq \bar C := Q\nu^{-1}(\theta/\nu)^{k_0}(1-\xi)^{-1}$, then
\[
C_{k+1} \leq \frac \xi{1-\xi}Q\nu^{-1}(\theta/\nu)^{k_0} + Q\nu^{-1}(\theta/\nu)^{k_0} = \bar{C}.
\]
Taking $C' = \max \{ C_k : 0\leq k\leq k_0\}$ and $C'' = \max(C,C')$, we get
$
D_k \leq C'' \nu^k
$
for all $k$.  Since $C''$ does not depend on $x,z,n,k$, this completes the proof of Lemma \ref{lem:Grassmann-contract}.
Combining Lemmas \ref{lem:Grass-Holder} and \ref{lem:Grassmann-contract} gives the Bowen property for $\phigeo$ on $\GGG$, just as in Lemma \ref{bowenpropBV}.

\subsection*{Acknowledgments} We would like to thank Keith Burns, Martin Sambarino, Tianyu Wang, Amie Wilkinson, and the anonymous referees for numerous helpful comments. We also thank the American Institute of Mathematics, where some of this work was completed as part of a SQuaRE.

\bibliographystyle{plain}
\bibliography{CFT-references}

\end{document}